\documentclass{amsart}
\usepackage{amsmath,amssymb,bm}
\usepackage{amsmath}
\usepackage{amssymb}
\usepackage{amscd}
\usepackage{amsthm}
\usepackage[dvips]{graphicx}

\theoremstyle{definition}
\newtheorem{defn}{Definition}[section]
\newtheorem{example}[defn]{Example}
\newtheorem{rem}[defn]{Remark}
\theoremstyle{plain}
\newtheorem{thm}[defn]{Theorem}
\newtheorem{prop}[defn]{Proposition}
\newtheorem{lem}[defn]{Lemma}
\newtheorem{cor}[defn]{Corollary}

\newtheorem{claim}[defn]{Claim}

\newcommand{\im}{\operatorname{Im}}

\newcommand{\KH}{\operatorname{KH}}

\numberwithin{equation}{section}
\title{The maximal degree of the Khovanov homology of a cable link}
\author{Keiji Tagami}
\date{\today}
\address{
Department of Mathematics,
Tokyo Institute of Technology,
Oh-okayama, Meguro, Tokyo 152-8551, Japan
}
\email{tagami.k.aa@m.titech.ac.jp}
\begin{document}
\maketitle
\begin{abstract}
In this paper, we study the Khovanov homology of cable links. 
We first estimate the maximal homological degree term of the Khovanov homology of the ($2k+1$, $(2k+1)n$)-torus link and give a lower bound of its homological thickness. 
Specifically, we show that the homological thickness of the ($2k+1$, $(2k+1)n$)-torus link is greater than or equal to $k^{2}n+2$. 
Next, we study the maximal homological degree of the Khovanov homology of the ($p$, $pn$)-cabling of any knot with sufficiently large $n$. 
Furthermore, we compute the maximal homological degree term of the Khovanov homology of such a link with even $p$. 
As an application we compute the Khovanov homology and the Rasmussen invariant of a twisted Whitehead double of any knot with sufficiently many twists. 
\end{abstract}
\section{Introduction}\label{intro}
A knot is an embedding of a circle into the $3$-sphere. 
A link is an embedding of a disjoint union of finitely many circles into the $3$-sphere. \par
In \cite{khovanov1}, for each link $L$, Khovanov defined a graded chain complex whose graded Euler characteristic is equal to the Jones polynomial of $L$. 
Its homology group is a link invariant and called the Khovanov homology. 
Khovanov homology has two gradings, homological degree $i$ and $q$-grading $j$.
In this paper, we denote the homological degree $i$ term of the Khovanov homology of $L$ by $\KH^{i} (L)$ 
and denote the homological degree $i$ and $q$-grading $j$ term of the Khovanov homology of $L$ by $\KH^{i,j}(L)$. 
\par
The $(p,q)$-cabling $K(p, q)$ of a knot $K$ is the satellite link with companion $K$ and pattern the $(p, q)$-torus link $T_{p, q}$. 
The Alexander polynomial of a cable link satisfies the following formula (see \cite{knot-gtm}). 
\begin{center}
$\Delta _{K(p,q)}(t)=\Delta _{K}(t^{p})\Delta _{T_{p, q}}(t). $
\end{center}
The Jones polynomial of a cabling of $K$ is expressed in terms of the colored Jones polynomial of $K$. 
Indeed, the colored Jones polynomial has a cabling formula (for example, see \cite{KM:1991}). 
However, there are few works about the Khovanov homology (which is a categorification of the Jones polynomial) of cable links. 
The $(2k, 2kn)$-torus link $T_{2k, 2kn}$ can be regarded as the $(2k, 2kn)$-cabling of the unknot and Sto{\v s}i{\'c} \cite{stosic2} showed that the maximal homological degree of the Khovanov homology of $T_{2k,2kn}$ is $2k^{2}n$ (Theorem~$\ref{thm1}$). Moreover, he computed the homological degree $2k^{2}n$ term (see Theorem~$\ref{thm2}$). 
\par
In this paper, we consider the $(p, pn)$-cabling of any knot. 
Our main results are Theorems~$\ref{newmainthm}$ and $\ref{newmainthm2}$ below. 
\par
We first determine the maximal homological degree of the Khovanov homology of the $(2k+1, (2k+1)n)$-torus link $T_{2k+1, (2k+1)n}$ by Sto{\v s}i{\'c}'s method. 
In addition, we determine the dimension of the maximal Khovanov homology of such a link. 
\begin{thm}\label{newmainthm}
Let k and n be positive integers. 
Denote the $(2k+1, (2k+1)n)$-torus link by $T_{2k+1, (2k+1)n}$. 
Assume that its orientation is given by the closure of the braid $(\sigma _{1}\cdots\sigma _{2k})^{(2k+1)n}$ with all crossings positive, where the $\sigma_{i}$ are the standard generators of the braid group $B_{2k+1}$. 
Then, for $i>2k(k+1)n$, we have 
\begin{center}
$\KH^{i}(T_{2k+1, (2k+1)n})=0$. 
\end{center}
On the other hand, 
\begin{align*}
\dim_{\mathbf{Q}}\KH^{2k(k+1)n}(T_{2k+1, (2k+1)n})=\begin{pmatrix}
2k+2\\
k+1
\end{pmatrix}.
\end{align*}
Moreover, for $i=0, \dots, k+1$, we have 
\begin{center}
$\KH^{2k(k+1)n, 6k(k+1)n+1-2i}(T_{2k+1, (2k+1)n})\neq 0$. 
\end{center}
\end{thm}
From Theorem~$\ref{newmainthm}$, we obtain the following. 
\begin{cor}\label{i_max1}
Let k and n be positive integers. 
Then we have 
\begin{center}
$\max \{i\in\mathbf{Z}|\KH^{i}(T_{2k+1, (2k+1)n})\neq 0\}=2k(k+1)n$. 
\end{center}
\end{cor}
Moreover, we also obtain an estimation of the homological thickness of $T_{2k+1, (2k+1)n}$ (see Corollary~$\ref{cor_torus_thick}$).  
\par
Next we consider the $(p, pn)$-cabling $K(p, pn)$ of any oriented knot $K$. 
Assume that each component of $K(p, pn)$ has an orientation induced by $K$, that is, each component of $K(p, pn)$ is homologous to $K$ in the tubular neighborhood of $K$. 
For such a link, we obtain an analog of Theorem~$\ref{newmainthm}$. 
\begin{thm}\label{newmainthm2}
Let $K$ be an oriented knot and $D$ be a diagram of $K$with $l_{+}$ positive crossings and $l_{-}$ negative crossings. 
Put $l=l_{+}+l_{-}$ and $f=l_{+}-l_{-}$. 
Then for $n\geq l $ and any positive integer $k$, we obtain the following:  
\begin{align*}
\max\{i\in\mathbf{Z}|\KH^{i}(K(2k, 2k(n+f)))\neq 0\}=2k^{2}(n+f). 
\end{align*}
In addition, if $n>l$, we determine the dimension of the maximal Khovanov homology of the link:  
\begin{align*}
\dim_{\mathbf{Q}} \KH^{2k^{2}(n+f)}(K(2k, 2k(n+f)))=\begin{pmatrix}
2k\\
k
\end{pmatrix}.
\end{align*}
Moreover, 
for $n>l$ and $i=0, \dots, k$, we have 
\begin{align*}
\KH^{2k^{2}(n+f), 6k^{2}(n+f)-2i}(K(2k, 2k(n+f)))\neq 0. 
\end{align*}
\end{thm}
Corollary~$\ref{i_max1}$ and the first claim of Theorem~$\ref{newmainthm2}$ imply a relation between the number of full twists and the maximal degree of the Khovanov homology.  

We also estimate the maximal homological degree of the Khovanov homology of the $(2k+1, (2k+1)n)$-cabling of any knot $K$. 
\begin{prop}\label{newmainprop}
Let $K$ be an oriented knot and $D$ be a diagram of $K$with $l_{+}$ positive crossings and $l_{-}$ negative crossings. 
Put $l=l_{+}+l_{-}$ and $f=l_{+}-l_{-}$. 
Then for $n\geq l $ and any positive integer $k$, we have the following:  
\begin{align*}
2k(k+1)(n+f)&\leq \max\{i\in\mathbf{Z}|\KH^{i}(K(2k+1, (2k+1)(n+f)))\neq 0\}\\
&\leq 2k(k+1)(n+f)+l_{+}. 
\end{align*}
\end{prop}
\par
As an application, we can give a computation of the Khovanov homology of a twisted Whitehead double of any knot with sufficiently many twists (Proposition~$\ref{lem4}$), since a cable link is obtained from such a knot by smoothing at a crossing. 
Moreover we compute the Rasmussen invariant $s$ (\cite{rasmussen1}) of such a knot (Corollary~$\ref{ras_double}$). 
\par
The paper is organized as follows: 
In Section~$\ref{def}$, we recall the definition of Khovanov homology and our main tools. 
In Sections~$\ref{main}$ and $\ref{main2}$, we prove Theorems~$\ref{newmainthm}$ and $\ref{newmainthm2}$, and Proposition~$\ref{newmainprop}$. 
In Section~$\ref{apply}$, we present our results on Whitehead doubles. 
Section~$\ref{appendix}$ contains the proofs of several technical results. 
\section{Khovanov homology}\label{def}
\subsection{The definition of Khovanov homology}
In this subsection, we recall the definition of the (rational) Khovanov homology. 
Let $L$ be an oriented link. 
Take a diagram $D$ of $L$ and an ordering of the crossings of $D$. 
For each crossing of $D$, we define a {\it $0$-smoothing} and a {\it $1$-smoothing} as in Figure~$\ref{smoothing}$. 
A smoothing of $D$ is a diagram where each crossing of $D$ is changed by either $0$-smoothing or $1$-smoothing. 
\begin{figure}[!h]
\begin{center}
\includegraphics[scale=0.7]{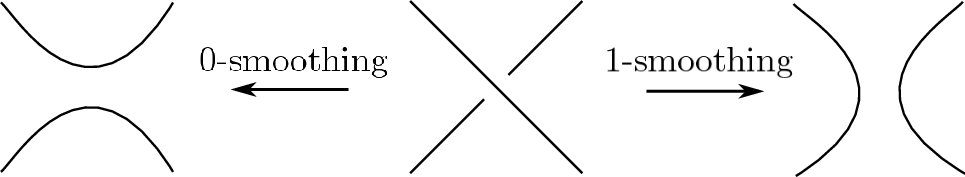}
\end{center}
\caption{0-smoothing and 1-smoothing. }
\label{smoothing}
\end{figure}
Let $n$ be the number of the crossings of $D$. Then $D$ has $2^{n}$ smoothings. 
By using the given ordering of the crossings of $D$, we have a natural bijection between the set of smoothings of $D$ and the set $\{0, 1\}^{n}$, where, to any $\varepsilon =(\varepsilon _{1}, \dots, \varepsilon _{n})\in \{0, 1\}^{n}$, we associate the smoothing $D_{\varepsilon }$ where the $i$-th crossing of $D$ is $\varepsilon_{i}$-smoothed. 
Each smoothing $D_{\varepsilon }$ is a collection of disjoint circles. 
\par
Let $V$ be a graded free $\mathbf{Q}$-module generated by $1$ and $X$ with $\operatorname{deg}(1)=1$ and $\operatorname{deg}(X)=-1$. 
Let $k_{\varepsilon }$ be the number of the circles of the smoothing $D_{\varepsilon }$. 
Put $M_{\varepsilon }=V^{\otimes k_{\varepsilon }}$. 
The module $M_{\varepsilon }$ has a graded module structure, that is, for $v=v_{1}\otimes\cdots\otimes v_{k_{\varepsilon }}\in M_{\varepsilon }$, $\deg(v):=\deg(v_{1})+\cdots+\deg(v_{k_{\varepsilon }})$. 
Then define 
\begin{align*}
C^{i}(D)&:=\bigoplus_{|\varepsilon |=i }M_{\varepsilon }\{i\},  
\end{align*}
where $|\varepsilon |=\sum_{i=1}^{m}\varepsilon _{i}$. 
Here, $M_{\varepsilon}\{i\}$ denotes $M_{\varepsilon}$ with its gradings shift by $i$ (for a graded module $M=\bigoplus_{j\in\mathbf{Z}}M^{j}$ and an integer $i$, we define the graded module $M\{i\}=\bigoplus_{j\in\mathbf{Z}}M\{i\}^{j}$ by $M\{i\}^{j}=M^{j-i}$). 
\par
The differential map $d^{i}\colon C^{i}(D)\rightarrow C^{i+1}(D)$ is defined as follows. 
Fix an ordering of the circles for each smoothing $D_{\varepsilon }$ and associate the $i$-th tensor factor of $M_{\varepsilon }$ to the $i$-th circle of $D_{\varepsilon }$.  
Take elements $\varepsilon$ and $\varepsilon ' \in \{0, 1\}^{n}$ such that $\varepsilon _{j}=0$ and $\varepsilon' _{j}=1$ for some $j$ and that $\varepsilon _{i}=\varepsilon' _{i}$ for any $i\neq j$. 
For such a pair $(\varepsilon , \varepsilon ')$, we will define a map $d_{\varepsilon \rightarrow \varepsilon '}\colon M_{\varepsilon }\rightarrow M_{\varepsilon '}$. 
\par
In the case where two circles of $D_{\varepsilon }$ merge into one circle of $D_{\varepsilon' }$,  the map $d_{\varepsilon \rightarrow \varepsilon '}$ is the identity on all factors except the tensor factors corresponding to the merged circles where it is a multiplication map $m\colon V\otimes V\rightarrow V$ given by: 
\begin{center}
$m(1\otimes 1)=1$,\  $m(1\otimes X)=m(X\otimes 1)=X$,\  $m(X\otimes X)=0$. 
\end{center}
\par
In the case where one circle of $D_{\varepsilon }$ splits into two circles of $D_{\varepsilon' }$,  the map $d_{\varepsilon \rightarrow \varepsilon '}$ is the identity on all factors except the tensor factor corresponding to the split circle where it is a comultiplication map $\Delta \colon V\rightarrow V\otimes V$ given by:
\begin{center}
$\Delta (1)=1\otimes X+X\otimes 1$,\  $\Delta (X)=X\otimes X$. 
\end{center}
\par
If there exist distinct integers $i$ and $j$ such that $\varepsilon _{i}\neq\varepsilon '_{i}$ and that $\varepsilon _{j}\neq\varepsilon '_{j}$, then define $d_{\varepsilon \rightarrow \varepsilon '}=0$. 
\par
In this setting, we define a map $d^{i}\colon C^{i}(D)\rightarrow C^{i+1}(D)$ by $\sum_{|\varepsilon|=i}d_{\varepsilon}^{i}$, where $d_{\varepsilon}^{i}\colon M_{\varepsilon}\rightarrow C^{i+1}(D)$ is defined by 
\begin{align*}
d^{i}(v):=\sum_{|\varepsilon'|=i+1 }(-1)^{l(\varepsilon, \varepsilon' )}d_{\varepsilon \rightarrow \varepsilon '}(v).   
\end{align*}
Here $v\in M_{\varepsilon }\subset C^{i}(D)$ and $l(\varepsilon, \varepsilon')$ is the number of $1$'s in front of (in our order) 
the factor of $\varepsilon$ which is different from $\varepsilon'$. 
\par
We can check that ($C^{i}(D)$, $d^{i}$) is a cochain complex and we denote its $i$-th homology group by $H^{i}(D)$. 
We call these the {\it unnormalized homology groups} of $D$. 
Since the map $d^{i}$ preserves the grading of $C^{i}(D)$, the group $H^{i}(D)$ has a graded structure $H^{i}(D)=\bigoplus_{j\in\mathbf{Z}}H^{i,j}(D)$ induced by that of $C^{i}(D)$. 
For any link diagram $D$, we define its Khovanov homology $\KH^{i, j}(D)$ by 
\begin{center}
$\KH^{i, j}(D)=H^{i+n_{-}, j-n_{+}+2n_{-}}(D)$, 
\end{center}
where $n_{+}$ and $n_{-}$ are the number of the positive and negative crossings of $D$, respectively. 
The grading $i$ is called the homological degree and $j$ is called the $q$-grading. 
\par
Let $D$ and $D'$ be link diagrams. The diagram $D$ is equivalent to $D'$ if $D'$ is obtained from $D$ by the Reidemeister moves (see Figure~$\ref{equivalent}$) and isotopies of the plane. 
It is known that two diagrams $D$ and $D'$ are diagrams of the same link if and only if $D$ is equivalent to $D'$. 
\begin{figure}[!h]
\begin{center}
\includegraphics[scale=0.4]{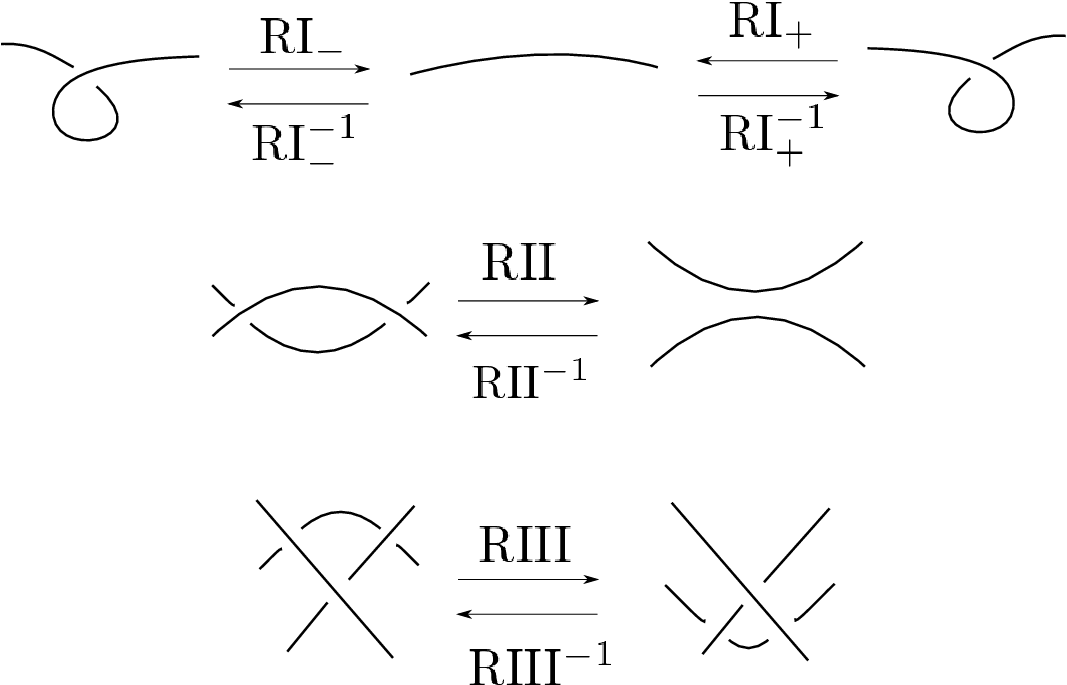}
\end{center}
\caption{Reidemeister moves. }
\label{equivalent}
\end{figure}
\begin{thm}[\cite{Bar-Natan-1}, \cite{khovanov1}]
Let $L$ be an oriented link and $D$ be a diagram of $L$. 
If $D'$ is equivalent to $D$, the homology groups $\KH(D)$ and $\KH(D')$ are isomorphic.  
In this sense, we can denote $\KH(D)$ by $\KH(L)$. 
Moreover, the graded Euler characteristic of the homology $\KH(L)$ equals the Jones polynomial of $L$, that is, 
\begin{align*}
V_{L}(t)=(q+q^{-1})^{-1}\sum_{i, j\in\mathbf{Z}}(-1)^{i}q^{j}\dim_{\mathbf{Q}}{\KH^{i, j}(L)}\Big|_{q=-t^{\frac{1}{2}}}, 
\end{align*}
where $V_{L}(t)$ is the Jones polynomial of $L$.  
\end{thm}
%
%
%
%
%
\subsection{Main tools}
Our main tools are the following (Theorems~$\ref{viro}$ and $\ref{spect}$ and Proposition~$\ref{4.3}$). 
\subsubsection{A long exact sequence}
\par
Let $D$ be a link diagram and $D_{i}$ be a diagram obtained from $D$ by $i$-smoothing at a crossing of $D$ (see Figure~$\ref{smoothing2}$). 
The following exact sequence was introduced in \cite{viro1} (see also \cite{viro2}). 
\begin{figure}[!h]
\begin{center}
\includegraphics[scale=0.6]{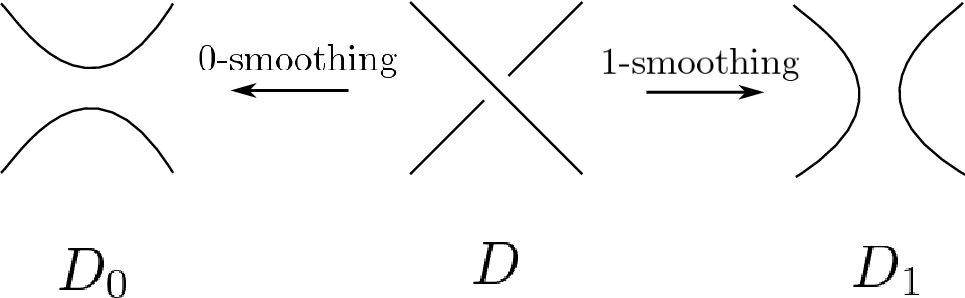}
\end{center}
\caption{$D$, $D_{0}$ and $D_{1}$. }
\label{smoothing2}
\end{figure}
\begin{thm}[\cite{viro1}]\label{viro}
There is a long exact sequence of the unnormalized homology groups: 
\begin{center}
$\cdots\to H^{i-1, j-1}(D_{1})\to H^{i, j}(D)\to H^{i, j}(D_{0})\to H^{i, j-1}(D_{1})\to\cdots$. 
\end{center}
\end{thm}
\par
\subsubsection{Lee homology}
Let $L$ be an oriented link. 
By $\operatorname{Lee}^{i}(L)$, we denote the homological degree $i$ term of the Lee homology of $L$ (for detail, see \cite{lee1}). 
\begin{thm}[\cite{lee1}]\label{spect}
There is a spectral sequence whose $E_{\infty}$-page is the Lee homology and $E_{2}$-page is the Khovanov homology.
\end{thm}
\begin{prop}[{\cite[Proposition~$4.3$]{lee1}}] \label{4.3}
Let $L$ be an oriented link with $n$ components, $S_{1}, \dots, S_{n}$. 
Then we have 
\begin{align*}
\dim_{\mathbf{Q}}(\operatorname{Lee}^{i}(L))=2\|\{E\subset \{2,\dots,n\}\mid\sum_{j\in E, k\notin E}2\operatorname{lk}(S_{j}, S_{k})=i\}\|, 
\end{align*}
where $\operatorname{lk}(S_{j}, S_{k})$ is the linking number of $S_{j}$ and $S_{k}$. 
\end{prop}
\section{The maximal degree of the Khovanov homology of the $(2k+1, (2k+1)n)$-torus link}\label{main}
In this section, we prove Theorem~$\ref{newmainthm}$ which has three claims. 
The first, second and third claims are Lemmas~$\ref{mainthm}$, $\ref{mainthm2}$ and $\ref{non-trivial1}$ below, respectively. 
We first introduce some results by Sto{\v s}i{\'c}. 
\par
\begin{defn}
We denote the $(p, q)$-torus link by $T_{p,q}$. 
Put $D_{p, q}=(\sigma _{1}\cdots\sigma _{p-1})^{q}$, where the $\sigma _{i}$ are the standard generators of the braid group $B_{p}$. 
The closure of the braid $D_{p, q}$ is a diagram of the $(p, q)$-torus link $T_{p,q}$. 
We give $T_{p,q}$ the downward orientation so that all crossings of $D_{p, q}$ are positive. 
\end{defn}
Sto{\v s}i{\'c} \cite{stosic2} showed the following results (Theorems~$\ref{thm1}$ and $\ref{thm2}$ and Corollaries~$\ref{i_max}$ and $\ref{torus_thick}$). 
\begin{thm}[{\cite[Theorem~$1$]{stosic2}}]\label{thm1}
Let $k$ and $n$ be positive integers. Then we have $\KH^{i}(T_{2k, 2kn})=0$ if $i>2k^{2}n$. 
\end{thm}
\begin{thm}[{\cite[Theorem~$3$]{stosic2}}]\label{thm2}
Let $k$ and $n$ be positive integers. Then we have 
\begin{align*}
\dim_{\mathbf{Q}} \KH^{2k^{2}n}(T_{2k, 2kn})=
\begin{pmatrix}
2k\\
k
\end{pmatrix}. 
\end{align*}
Moreover, we obtain 
\begin{align*}
\dim_{\mathbf{Q}}\KH^{2k^{2}n, 6k^{2}n-2i}(T_{2k, 2kn})=
\begin{cases}
\begin{pmatrix}
2k\\
k-i
\end{pmatrix}
-
\begin{pmatrix}
2k\\
k-i-1
\end{pmatrix}& \text{if\ }i=0, \dots, k, \\
\ &\ \\
\ \ \ 0 &\text{otherwise}.
\end{cases}
\end{align*}
\end{thm}
\par
\ 
\par
From the above results, we can determine the maximal homological degree of the Khovanov homology of the $(2k, 2kn)$-torus link. 
\begin{cor}[\cite{stosic2}]\label{i_max}
Let k and n be positive integers. Then we obtain $\max \{i\in\mathbf{Z}|\KH^{i}(T_{2k, 2kn})\neq 0\}=2k^{2}n$. 
\end{cor}
Moreover we can estimate the homological thickness of the $(2k, 2kn)$-torus link. 
\begin{cor}[{\cite[Corollary~$5$]{stosic2}}]\label{torus_thick}
The homological thickness $\operatorname{hw}(T_{2k, 2kn})$ of the $(2k$, $2kn)$-torus link is greater than or equal to $k(k-1)n+2$, where the homological thickness $\operatorname{hw}(L)$ of a link $L$ is defined as $(\max \{j-2i|KH^{i, j}(L)\neq 0\}-\min \{j-2i|KH^{i, j}(L)\neq 0\})/2+1$. 
\end{cor}
The homological thickness of a link estimates a distance between the link and an alternating link as follows. 
A link is {\it $k$-almost alternating} if it has a reduced diagram which can be alternating after $k$ crossing changes and no diagram which can be alternating after $k-1$ or less crossing changes (see \cite{almost_alternating}).  
Then we have the following results. 
\begin{thm}[{\cite[Theorem~$8$]{spanning2_kh}}]\label{dalt_hw}
Let $L$ be a $k$-almost alternating link. Then we obtain 
\begin{center}
$k\geq \operatorname{hw}(L)-2. $
\end{center}
\end{thm}
\begin{rem}
From Corollary~$\ref{torus_thick}$ and Theorem~$\ref{dalt_hw}$, the $(2k, 2kn)$-torus link has no diagram which is alternating after $k(k-1)n-1$ or less crossing changes. 
\end{rem}
\par
Theorem~$\ref{newmainthm}$ can be regarded as an analog of Theorems~$\ref{thm1}$ and $\ref{thm2}$ and Corollary~$\ref{i_max}$. 
Theorem~$\ref{newmainthm}$ follows from Lemmas~$\ref{mainthm}$, $\ref{mainthm2}$ and $\ref{non-trivial1}$ below. 
We will prove these Lemmas. 
\begin{lem}\label{mainthm}
Let k and n be positive integers. Then we have $\KH^{i}(T_{2k+1, (2k+1)n})=0$ if $i>2k(k+1)n$. 
\end{lem}
\begin{proof}
In Section $\ref{main2}$, we prove Proposition~$\ref{newmainprop}$, which implies Lemma~$\ref{mainthm}$. 
\end{proof}
\par
Next we introduce Lemma~$\ref{mainthm2}$. 
We can consider Lemma~$\ref{mainthm2}$ to be an analog of the first claim of Theorem~$\ref{thm2}$. 
\begin{lem}\label{mainthm2}
Let k and n be positive integers. Then we have 
\begin{align*}
\dim_{\mathbf{Q}}\KH^{2k(k+1)n}(T_{2k+1, (2k+1)n})=\begin{pmatrix}
2k+2\\
k+1
\end{pmatrix}.
\end{align*}
\end{lem}
To prove Lemma~$\ref{mainthm2}$, we use the same notation as Sto{\v s}i{\'c}'s in \cite{stosic}. \par
\begin{defn}[\cite{stosic}]
Let $K$ be any positive braid link, that is, $K$ has a diagram which is the closure of a positive braid. 
Let $D$ be its diagram which is the closure of a positive braid with $p$ strands. 
The crossing $c$ of $D$ is of the type $\sigma _{i}$ $(i<p)$ if it corresponds to the generator $\sigma _{i}$ in the positive braid. 
Let $c^{i}_{1}, \dots, c^{i}_{l_{i}}$ be of the type $\sigma _{i}$ crossings of $D$ and order them from top to bottom in the positive braid. 
Then we denote the crossing $c^{i}_{\alpha }$ by $(i, \alpha )$, where $1\leq i\leq p$ and $1\leq \alpha \leq l_{i}$. 
\par
Let $3\leq p\leq q$. 
Let $E_{p, q}^{1}$ and $D_{p, q}^{1}$ be the diagrams obtained from $D_{p, q}$ by $1$-smoothing and  $0$-smoothing at the crossing $(p-1, 1)$ of $D_{p, q}$, respectively. 
We continue the same process. 
Let $E_{p, q}^{2}$ and $D_{p, q}^{2}$ be the diagrams obtained from $D_{p, q}^{1}$ by $1$-smoothing and  $0$-smoothing at the crossing $(p-2, 1)$ of $D_{p, q}^{1}$ respectively. 
Repeating this process $p-1$ times, that is, for any $k=1, \dots, p-1$, 
let $E_{p, q}^{k}$ and $D_{p, q}^{k}$ be the diagrams obtained from $D_{p, q}^{k-1}$ by $1$-smoothing and  $0$-smoothing at the crossing $(p-k, 1)$ of $D_{p, q}^{k-1}$ respectively. 
Note that $D_{p, q}^{0}=D_{p,q}$ and that $D_{p,q}^{p-1}=D_{p,q-1}$. 
For example, see Figure~$\ref{the_diagrams}$. 
\par
We define $H^{i, j}(E^{k}_{p,q}):=H^{i, j}(\overline{E^{k}_{p,q}})$ and $H^{i, j}(D^{k}_{p,q}):=H^{i, j}(\overline{D^{k}_{p,q}})$, where $\overline{E^{k}_{p,q}}$ and $\overline{D^{k}_{p,q}}$ are the closure of $E^{k}_{p,q}$ and $D^{k}_{p,q}$, respectively. 
%
\begin{figure}[!h]
\begin{center}
\includegraphics[scale=0.45]{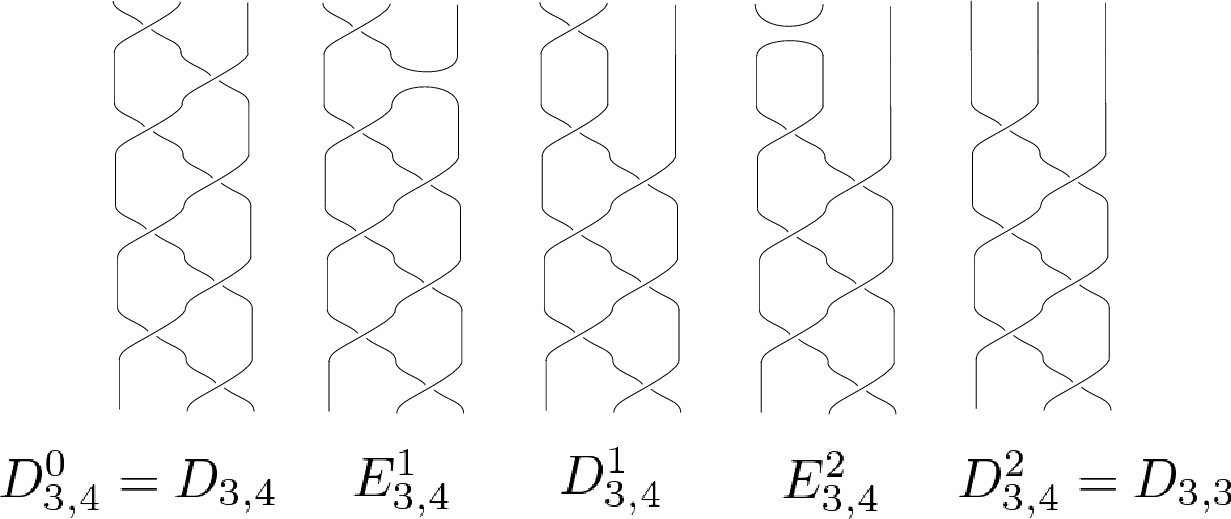}
\end{center}
\caption{$D_{3,4}=D_{3, 4}^{0}$, $E_{3,4}^{1}$, $D_{3,4}^{1}$, $E_{3,4}^{2}$, and $D_{3,4}^{2}=D_{3,3}$. }
\label{the_diagrams}
\end{figure}
\end{defn}
%
%
%
\par
From Theorem~$\ref{viro}$, we have the following long exact sequence for $k=1, \dots, p-1$: 
\begin{align}
\cdots\to H^{i-1, j-1}(E_{p, q}^{k})\to H^{i, j}(D_{p, q}^{k-1})\to H^{i, j}(D_{p, q}^{k})\to H^{i, j-1}(E_{p, q}^{k})\to\cdots. \label{long}
\end{align}
We use the following lemma, whose proof will be given in Section~$\ref{appendix}$. 
\begin{lem}\label{lem1}
Let k and n be positive integers. Then we have 
\begin{center}
$H^{2k(k+1)n}(D_{2k+1, (2k+1)n-1})=0$. 
\end{center}
\end{lem}
\begin{proof}[Proof of Lemma~$\ref{mainthm2}$]
To prove this lemma, it is sufficient to prove the following: 
\begin{align}
\dim_{\mathbf{Q}} H^{2k(k+1)n}(D_{2k+1, (2k+1)n}^{l})=2
\begin{pmatrix}
2k+1-l\\
k+1
\end{pmatrix}, \label{a}
\end{align}
where $0\leq l\leq 2k$ (for convenience, we define 
$\begin{pmatrix}
a\\
b
\end{pmatrix}=0$ if $0\leq a<b$). 
Indeed, if we put $l=0$ in $(\ref{a})$ then we have 
\begin{align*}
\dim_{\mathbf{Q}}\KH^{2k(k+1)n}(T_{2k+1, (2k+1)n})&=\dim_{\mathbf{Q}} H^{2k(k+1)n}(D_{2k+1, (2k+1)n}^{0})\\
&=2
\begin{pmatrix}
2k+1\\
k+1
\end{pmatrix}
=\begin{pmatrix}
2k+2\\
k+1
\end{pmatrix}. 
\end{align*}
%
To prove $(\ref{a})$, we use induction on $k$. 
\par
For $k=1$, we need to compute $H^{4n}(D_{3, 3n})$, $H^{4n}(D_{3, 3n}^{1})$ and $H^{4n}(D_{3, 3n}^{2})$. 
Note that $D_{3, 3n}^{2}=D_{3, 3n-1}$. 
The Khovanov homology of the $(3, q)$-torus link is known (for example, see \cite[Theorem~$8$]{stosic2} or \cite[Theorem~$3.1$]{turner}). 
In particular, 
\begin{align*}
\dim_{\mathbf{Q}} H^{4n}(D^{2}_{3,3n})=\dim_{\mathbf{Q}} H^{4n}(D_{3,3n-1})=0
\end{align*}
and 
\begin{align*}
\dim_{\mathbf{Q}} H^{4n}(D^{0}_{3,3n})=\dim_{\mathbf{Q}} H^{4n}(D_{3,3n})=6. 
\end{align*}
Next we compute the Khovanov homology of $D^{1}_{3,3n}$. 
We have the following long exact sequence: 
\begin{align}
\cdots\to H^{4n-1, j}(D_{3, 3n}^{2})&\to H^{4n-1, j-1}(E_{3, 3n}^{2})\to H^{4n, j}(D_{3, 3n}^{1})\to 0. \label{long1}
\end{align}
We can check that the closure of $E_{3, 3n}^{2}$ is a diagram of the unknot and that it has $4n-1$ negative crossings and $2n-1$ positive crossings. 
From the definition of the Khovanov homology, we obtain 
\begin{align*}
H^{4n-1, j-1}(E_{3, 3n}^{2})=\KH^{0, j-6n}(U)
=\begin{cases}
\mathbf{Q}& \text{if\ }j=6n\pm1,\\
0& \text{if\ }j\neq 6n\pm 1, 
\end{cases}
\end{align*}
where $U$ is the unknot. \par
Hence, from $(\ref{long1})$, we have 
\begin{center}
$\dim_{\mathbf{Q}} H^{4n}(D_{3, 3n}^{1})\leq  2$. 
\end{center}
On the other hand, from Proposition~$\ref{4.3}$, the dimension of $\operatorname{Lee}^{4n}(D_{3, 3n}^{1})$ is $2$. 
Since there is a spectral sequence whose $E_{\infty}$-page is the Lee homology and $E_{2}$-page is the Khovanov homology (Theorem~$\ref{spect}$), we have 
\begin{align*}
\dim_{\mathbf{Q}} H^{4n}(D_{3, 3n}^{1})\geq 2. 
\end{align*}
Hence we obtain 
\begin{align*}
\dim_{\mathbf{Q}} H^{4n}(D_{3, 3n}^{1})=2. 
\end{align*}
\par
Suppose that $(\ref{a})$ is true for $1, \dots, k-1$, that is, suppose that for $1\leq h< k$, $n>0$ and $l=0, \dots, 2h$, we have 
\begin{align}
\dim_{\mathbf{Q}} H^{2h(h+1)n}(D_{2h+1, (2h+1)n}^{l})=2
\begin{pmatrix}
2h+1-l\\
h+1
\end{pmatrix}. \label{a_suppose}
\end{align}
We will show that $(\ref{a})$ is true for $k$. 
For $l=0, \dots, 2k-1$, we obtain the following long exact sequence: 
\begin{multline}
\cdots\to H^{2k(k+1)n-1, j-1}(E_{2k+1, (2k+1)n}^{l+1})\xrightarrow{g^{l}_{j}} H^{2k(k+1)n, j}(D_{2k+1, (2k+1)n}^{l})\\
\xrightarrow{f^{l}_{j}} H^{2k(k+1)n, j}(D_{2k+1, (2k+1)n}^{l+1})\to \cdots . \label{kanzen}
\end{multline}
From the exact sequence $(\ref{kanzen})$, we obtain 
\begin{align}
&\sum_{j}\dim_{\mathbf{Q}} H^{2k(k+1)n, j}(D_{2k+1, (2k+1)n}^{l}) \label{add1} \\
&\leq \sum_{j}(\dim_{\mathbf{Q}} \im {g^{l}_{j}}+\dim_{\mathbf{Q}} \im {f^{l}_{j}}) \nonumber \\
&\leq \sum_{j} (\dim_{\mathbf{Q}} H^{2k(k+1)n-1, j-1}(E_{2k+1, (2k+1)n}^{l+1})+\dim_{\mathbf{Q}} H^{2k(k+1)n, j}(D_{2k+1, (2k+1)n}^{l+1}))\nonumber\\
&\leq \cdots \leq \nonumber\\
&\leq \sum_{j}\sum_{m=l+1}^{2k}\dim_{\mathbf{Q}} H^{2k(k+1)n-1, j-1}(E_{2k+1, (2k+1)n}^{m})\nonumber 
+\dim_{\mathbf{Q}} H^{2k(k+1)n}(D^{2k}_{2k+1, (2k+1)n})\nonumber. 
\end{align}
From Lemma~$\ref{lem1}$, we have $\dim_{\mathbf{Q}} H^{2k(k+1)n}(D_{2k+1, (2k+1)n-1})=0$. 
To compute $\dim_{\mathbf{Q}} H^{2k(k+1)n-1}(E_{2k+1, (2k+1)n}^{m})$, we consider the closure of $E_{2k+1, (2k+1)n}^{m}$. 
Note that the closure of $E_{2k+1, (2k+1)n}^{i}$ is equivalent to the closure of $D_{2k-1, (2k-1)n}^{i-2}$ for $i\geq 2$ (see Figure~$\ref{orientation1}$). 
We give the closure of $E_{2k+1, (2k+1)n}^{i}$ an orientation such that all crossings of the closure of $D_{2k-1, (2k-1)n}^{i-2}$ are positive. 
Then we can check that the closure of $E_{2k+1, (2k+1)n}^{i}$ has $4kn-1$ negative crossings. 
Hence for $i\geq 2$ we have
\begin{align*}
H^{2(k+1)kn-1}(E_{2k+1, (2k+1)n}^{i})=\KH^{2(k-1)kn}(D_{2k-1, (2k-1)n}^{i-2}). 
\end{align*}
Similarly, the closure of $E_{2k+1, (2k+1)n}^{1}$ is equivalent to the closure of $D_{2k-1, (2k-1)n}\sqcup\bigcirc $, where $\bigcirc $ is a circle in the plane (see Figure~$\ref{orientation2}$). 
We give the closure of $E_{2k+1, (2k+1)n}^{1}$ an orientation such that all crossings of the closure of $D_{2k-1, (2k-1)n}\sqcup\bigcirc $ are positive. 
Then we can check that the closure of $E_{2k+1, (2k+1)n}^{1}$ also has $4kn-1$ negative crossings. 
Hence we have 
\begin{align*}
H^{2(k+1)kn-1}(E_{2k+1, (2k+1)n}^{1})=\KH^{2(k-1)kn}(D_{2k-1, (2k-1)n}\sqcup\bigcirc ). 
\end{align*}
\begin{figure}[!h]
\begin{center}
\includegraphics[scale=0.46]{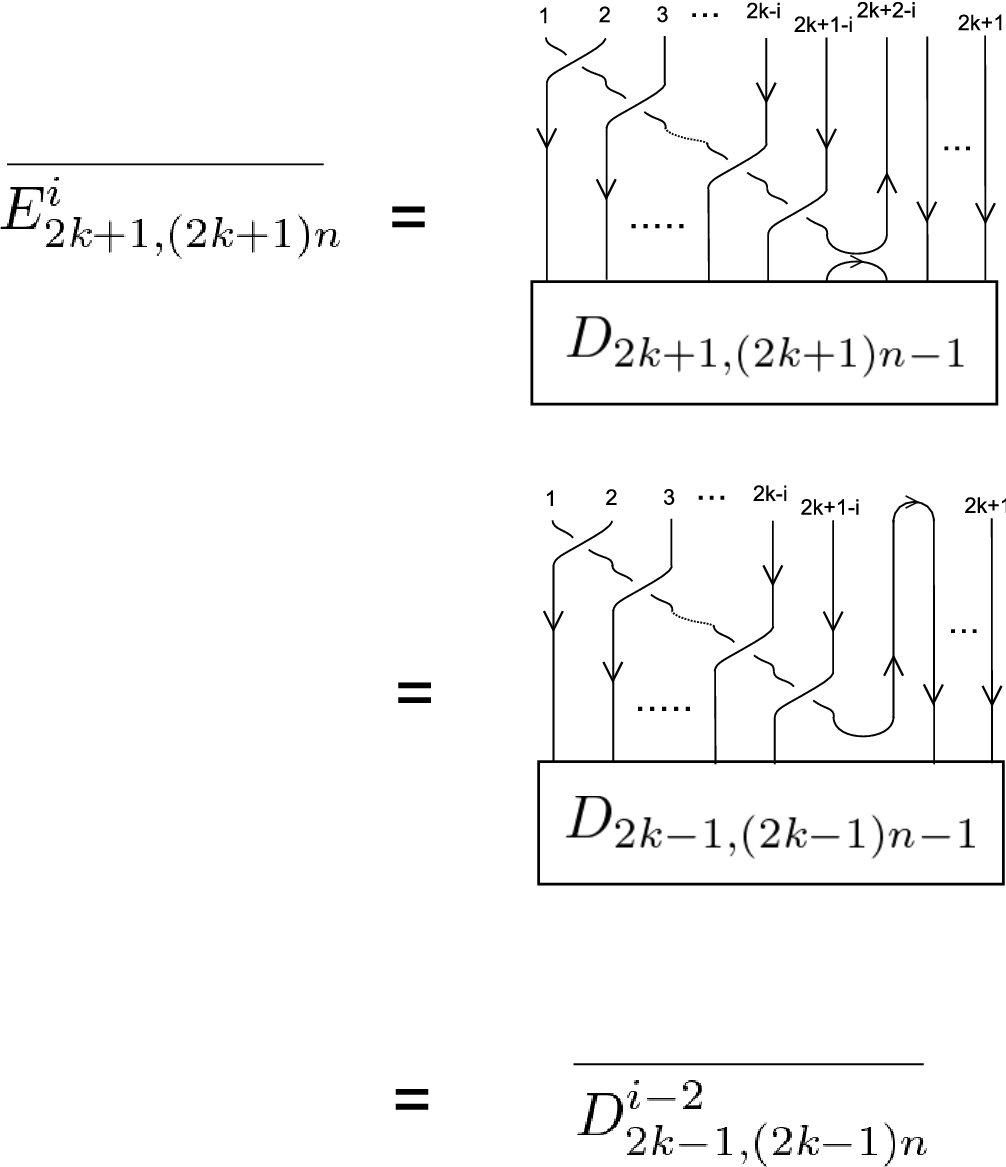}
\end{center}
\caption{The closure of $E_{2k+1, (2k+1)n}^{i}$ is equivalent to the closure of $D_{2k-1, (2k-1)n}^{i-2}$ for $i\geq 2$. }
\label{orientation1}
\end{figure}
\begin{figure}[!h]
\begin{center}
\includegraphics[scale=0.46]{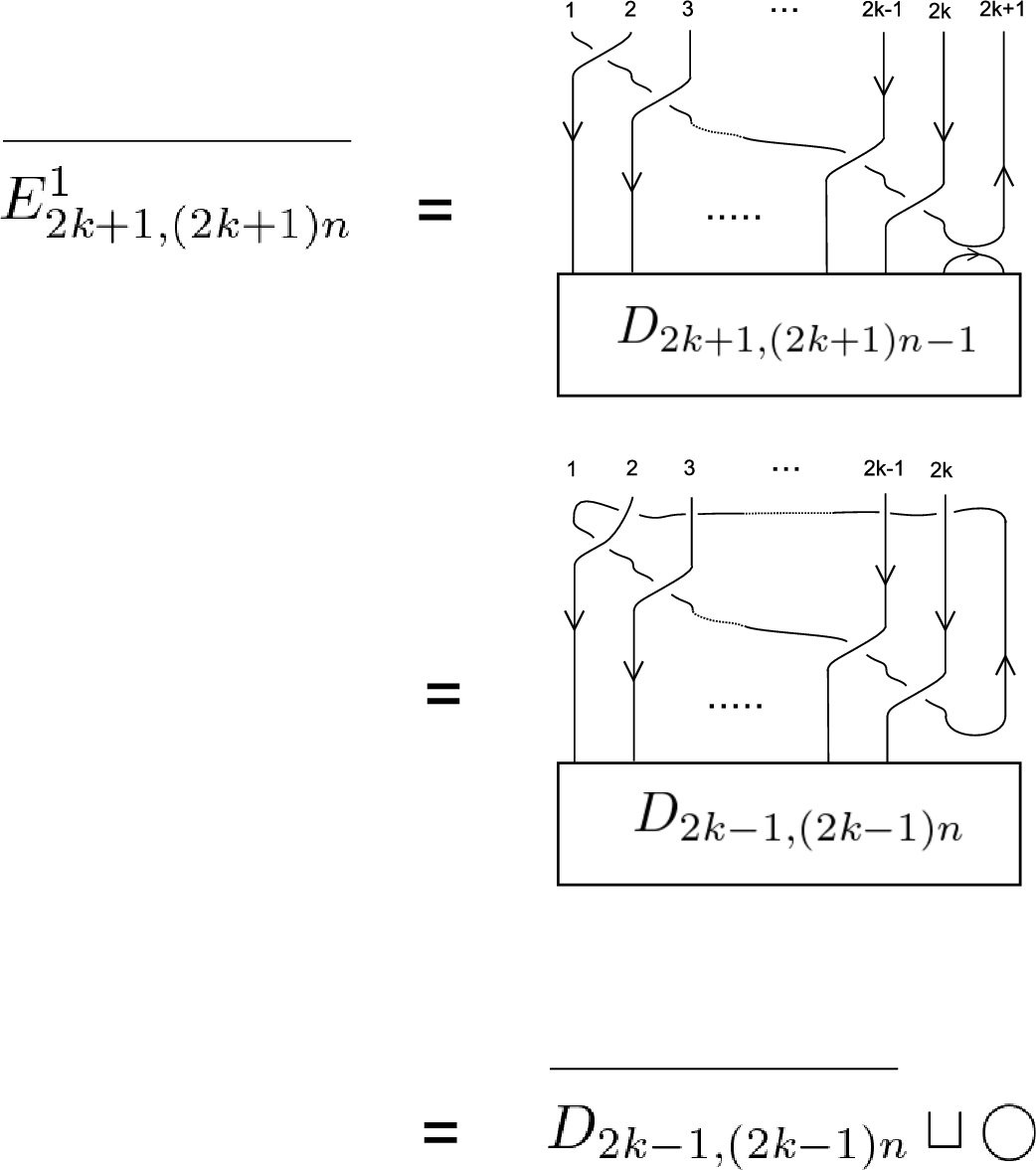}
\end{center}
\caption{The closure of $E_{2k+1, (2k+1)n}^{1}$ is equivalent to the closure of $D_{2k-1, (2k-1)n}\sqcup\bigcirc $. }
\label{orientation2}
\end{figure}
By the induction hypothesis $(\ref{a_suppose})$, we obtain  
\begin{align}
\dim_{\mathbf{Q}} H^{2(k+1)kn-1}(E_{2k+1, (2k+1)n}^{i})&=2
\begin{pmatrix}
2k+1-i\\
k
\end{pmatrix} \text{\ }(i\geq 2), \label{add2}
\end{align}
\begin{align}
\dim_{\mathbf{Q}} H^{2(k+1)kn-1}(E_{2k+1, (2k+1)n}^{1})&=2\times 2
\begin{pmatrix}
2k-1\\
k
\end{pmatrix}
=2
\begin{pmatrix}
2k\\
k
\end{pmatrix}. \label{add3}
\end{align}
From $(\ref{add1})$, $(\ref{add2})$ and $(\ref{add3})$, we obtain  
\begin{align}
\sum_{j}\dim_{\mathbf{Q}} H^{2k(k+1)n, j}(D_{2k+1, (2k+1)n}^{l})
\leq\sum_{m=l+1}^{2k}2\begin{pmatrix}
2k+1-m\\
k
\end{pmatrix}
=2\begin{pmatrix}
2k+1-l\\
k+1
\end{pmatrix}. \label{hutousiki1}
\end{align}
\par
Finally we will prove that the inequality in $(\ref{hutousiki1})$ is in fact an equality for $l=0, \dots, 2k$. 
%
At first, we consider the case where $l=0$. 
The dimension of $\operatorname{Lee}^{2k(k+1)n}(D_{2k+1, (2k+1)n})$ is $\begin{pmatrix}
2k+2\\
k+1
\end{pmatrix}$.
From Theorem~$\ref{spect}$, we have 
\begin{align*} 
\begin{pmatrix}
2k+2\\
k+1
\end{pmatrix}&=\dim_{\mathbf{Q}} \operatorname{Lee}^{2k(k+1)n}(D_{2k+1, (2k+1)n})\\
&\leq \dim_{\mathbf{Q}} H^{2k(k+1)n}(D_{2k+1, (2k+1)n})
\leq \begin{pmatrix}
2k+2\\
k+1
\end{pmatrix}. 
\end{align*}
This implies that we have the equality in $(\ref{hutousiki1})$ for $l=0$.  
Hence, for any $j\in \mathbf{Z}$ and $m=0, \dots, 2k-1$, the maps $g^{m}_{j}$ and $f^{m}_{j}$ in $(\ref{kanzen})$ are injective and surjective, respectively. 
In particular, we obtain 
\begin{align}
\dim_{\mathbf{Q}}\im g^{m}_{j}&=\dim_{\mathbf{Q}} H^{2k(k+1)n-1, j-1}(E_{2k+1, (2k+1)n}^{m+1}),\label{g}
\end{align}
\begin{align}
\dim_{\mathbf{Q}}\im f^{m}_{j}&=\dim_{\mathbf{Q}} H^{2k(k+1)n, j}(D_{2k+1, (2k+1)n}^{m+1}). \label{f}
\end{align}
%
%
From $(\ref{g})$ and $(\ref{f})$, we have the equality in $(\ref{hutousiki1})$ for $l=0, \dots, 2k$ and obtain 
\begin{align*}
\dim_{\mathbf{Q}} H^{2k(k+1)n}(D_{2k+1, (2k+1)n}^{l-1})
=2\begin{pmatrix}
2k+2-l\\
k+1
\end{pmatrix}.
\end{align*}
\end{proof}
%
%
%
%
%
The following lemma can be regarded as an analog of the second claim of Theorem~$\ref{thm2}$. 
\begin{lem}\label{non-trivial1}
For $i=0, \dots, k+1$, we have 
\begin{center}
$\KH^{2k(k+1)n, 6k(k+1)n+1-2i}(T_{2k+1, (2k+1)n})\neq 0$. 
\end{center}
\end{lem}
\begin{proof}
To prove this lemma, we use induction on $k$. \par
For $k=1$, it has already known that $\KH^{4n, 12n+1}(T_{3, 3n})$, $\KH^{4n, 12n-1}(T_{3, 3n})$ and $\KH^{4n, 12n-3}(T_{3, 3n})$ are not zero (see \cite[Theorem~$8$]{stosic2} or \cite[Theorem~$3.1$]{turner}). 
\par
Suppose that Lemma~$\ref{non-trivial1}$ is true for $1, \dots, k-1$, that is, suppose that for $1\leq h<k$, $n>0$ and $i=0, \dots, h+1$, we have 
\begin{align}
\KH^{2h(h+1)n, 6h(h+1)n+1-2i}(T_{2h+1, (2h+1)n})\neq 0. \label{nontri_suppose}
\end{align}
From the proof of Lemma~$\ref{mainthm2}$ (, recall that the inequality $(\ref{add1})$ is in fact an equality), we obtain 
\begin{align}
\dim_{\mathbf{Q}} H^{2k(k+1)n, j}(D_{2k+1, (2k+1)n})  
&=\sum_{m=1}^{2k}\dim_{\mathbf{Q}} H^{2k(k+1)n-1, j-1}(E_{2k+1, (2k+1)n}^{m}) \label{add4} \\
&\geq \dim_{\mathbf{Q}} H^{2k(k+1)n-1, j-1}(E_{2k+1, (2k+1)n}^{1})\nonumber \\
&\ \ \ +\dim_{\mathbf{Q}} H^{2k(k+1)n-1, j-1}(E_{2k+1, (2k+1)n}^{2})\nonumber. 
\end{align}
Note that the closure of $E_{2k+1, (2k+1)n}^{2}$ is equivalent to the closure of $D_{2k-1, (2k-1)n}$ (see Figure~$\ref{orientation1}$). 
We give the closure of $E^{2}_{2k+1, (2k+1)n}$ an orientation such that all crossings of the closure of $D_{2k-1, (2k-1)n}$ are positive. 
Then we can check that the closure of $E_{2k+1, (2k+1)n}^{2}$ has $4kn-1$ negative crossings and $2k(2k-1)n-1$ positive crossings. 
Similarly, the closure of $E_{2k+1, (2k+1)n}^{1}$ is equivalent to the closure of $D_{2k-1, (2k-1)n}\sqcup\bigcirc $, where $\bigcirc $ is a circle in the plane (see Figure~$\ref{orientation2}$). 
We give the closure of $E_{2k+1, (2k+1)n}^{1}$ an orientation such that all crossings of the closure of $D_{2k-1, (2k-1)n}\sqcup\bigcirc $ are positive. 
We can check that the closure of $E_{2k+1, (2k+1)n}^{1}$ has $4kn-1$ negative crossings and $2k(2k-1)n$ positive crossings. 
From $(\ref{add4})$, we have 
\begin{align*}
&\dim_{\mathbf{Q}} \KH^{2k(k+1)n, 6k(k+1)n+1-2i}(D_{2k+1, (2k+1)n})\\
&\ \ \ \ \geq\dim_{\mathbf{Q}} \KH^{2k(k-1)n, 6k(k-1)n+2-2i}(D_{2k-1, (2k-1)n}\sqcup\bigcirc )\\
&\ \ \ \ \ \ \ \ +\dim_{\mathbf{Q}} \KH^{2k(k-1)n, 6k(k-1)n+1-2i}(D_{2k-1, (2k-1)n}). 
\end{align*}
By the induction hypothesis $(\ref{nontri_suppose})$, the first term of the last expression is not zero for $i=1, \dots, k+1$ and the second term is not zero for $i=0, \dots, k$. 
\end{proof}
From Lemma~$\ref{non-trivial1}$, we obtain the following. 
\begin{cor}\label{cor_torus_thick}
The homological thickness $\operatorname{hw}(T_{2k+1, (2k+1)n})$ of the $(2k+1$, $(2k+1)n)$-torus link is greater than or equal to $k^{2}n+2$. 
\end{cor}
\begin{proof}
From Lemma~$\ref{non-trivial1}$, we have 
\begin{align*}
\KH^{2k(k+1)n, 6k(k+1)n+1-2(k+1)}(T_{2k+1, (2k+1)n})\neq 0. 
\end{align*}
In \cite{khovanov2}, Khovanov determines the homological degree $0$ term of the Khovanov homology of a positive link (see Theorem~$\ref{khovanov2}$ below). 
Note that, in \cite{khovanov2}, he denotes $\KH^{i, -j}$ by $\mathcal{H}^{i, j}$. \par
The closure of $D_{2k+1, (2k+1)n}$ is a positive diagram of $T_{2k+1, (2k+1)n}$. 
The number of its Seifert circles is $2k+1$ and the number of its crossings is $2k(2k+1)n$. 
From Theorem~$\ref{khovanov2}$, we have 
\begin{align*}
\KH^{0, 2k((2k+1)n-1)+1}(T_{2k+1, (2k+1)n})\neq 0. 
\end{align*}
Hence, by the definition of the homological thickness (cf. Corollary~$\ref{torus_thick}$), we obtain 
\begin{align*}
\operatorname{hw}(T_{2k+1, (2k+1)n})&\geq \frac{1}{2}(2k((2k+1)n-1)+1-2kn(k+1)-1+2(k+1))+1\\
&=k^{2}n+2. 
\end{align*}
\end{proof}
\begin{rem}
From Corollary~$\ref{cor_torus_thick}$ and Theorem~$\ref{dalt_hw}$, the $(2k+1, (2k+1)n)$-torus link has no diagram which is alternating after $k^{2}n-1$ or less crossing changes. 
\end{rem}
\begin{thm}[{\cite[Proposition~$6.1$]{khovanov2}}]\label{khovanov2}
Let $L$ be a positive link. 
Then $\KH^{i}(L)=0$ if $i<0$, 
\begin{align*}
\KH^{0, j}(L)=
\begin{cases}
\mathbf{Q} & \text{if\ } j=-s_{0}(D)+c+1\pm 1,\\
0& \text{otherwise}, 
\end{cases}
\end{align*}
and $\KH^{i, j}(L)=0$ if $i>0$ and $j<c-s_{0}(D)$, where $s_{0}(D)$ is the number of the Seifert circles and $c$ is the number of the crossings in a positive diagram $D$ of $L$. 
\end{thm}
\section{The maximal degree of the Khovanov homology of a cable link}\label{main2}
In this section, we prove Theorem~$\ref{newmainthm2}$ and Proposition~$\ref{newmainprop}$. 
Recall that Theorem~$\ref{newmainthm2}$ has three claims. 
These claims follow from Lemmas~$\ref{mainthm3}$, $\ref{mainthm4}$ and $\ref{non-trivial2}$ below, which are the first, second and third claims of Theorem~$\ref{newmainthm2}$, respectively. 
Hence, Theorem~$\ref{newmainthm2}$ immediately follows from these lemmas. 
Lemma~$\ref{mainthm3}$ also implies Proposition~$\ref{newmainprop}$. 
To prove these lemmas, we define some notations. 
\par
\begin{defn}
Let $K$ be an oriented knot and $D$ be a knot diagram of $K$ with writhe $f$. 
Denote the $(p, pn)$-cabling of the knot $K$ by $K(p, pn)$. 
Assume that each component of $K(p, pn)$ has an orientation induced by $K$, that is, each component of $K(p, pn)$ is homologous to $K$ in the tubular neighborhood of $K$. 
Let $D(p, q+pf)$ be the diagram depicted in Figure~$\ref{cable}$. The diagram $D(p, q+pf)$ is a diagram of the $(p, q+pf)$-cabling $K(p, q+pf)$ of $K$ (see Figure~$\ref{ex}$). 
Let $D^{m}(p, q+pf)$ and $E^{m}(p, q+pf)$ be the diagrams depicted in Figure~$\ref{E(mpq)}$. 
\par
\begin{figure}[!h]
\begin{center}
\includegraphics[scale=0.3]{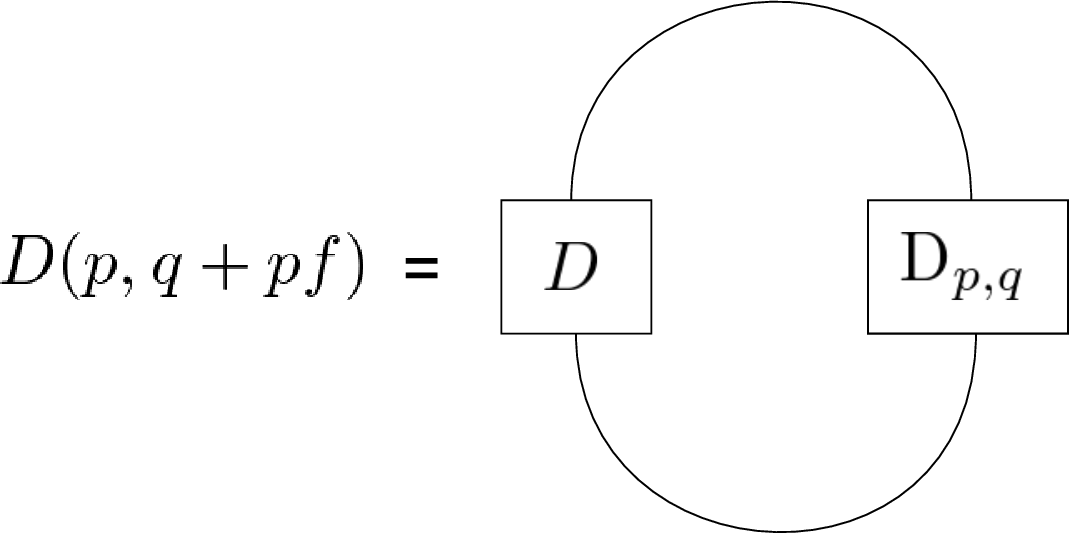}
\end{center}
\caption{The diagram $D(p, q+pf)$ is obtained from $p$-parallel of $D$ by adding $D_{p, q}$, where $f$ is the writhe of $D$.  The diagram $D(p, q+pf)$ is a diagram of the $(p, q+pf)$-cabling of $K$.}
\label{cable}
\end{figure}
\begin{figure}[!h]
\begin{center}
\includegraphics[scale=0.3]{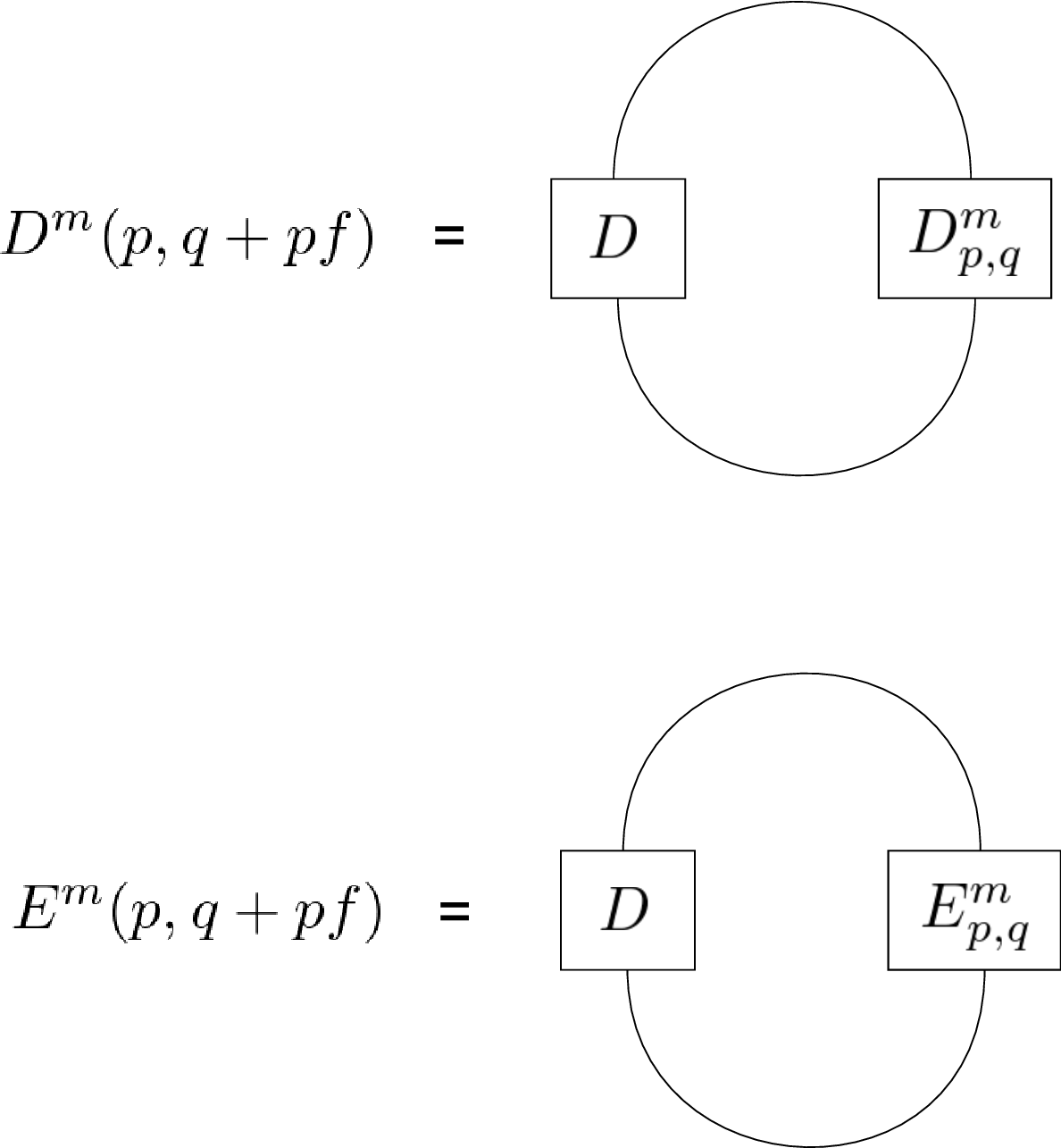}
\end{center}
\caption{The diagram $D^{m}(p, q+pf)$ and $E^{m}(p, q+pf)$. }
\label{E(mpq)}
\end{figure}
\begin{figure}[!h]
\begin{center}
\includegraphics[scale=0.35]{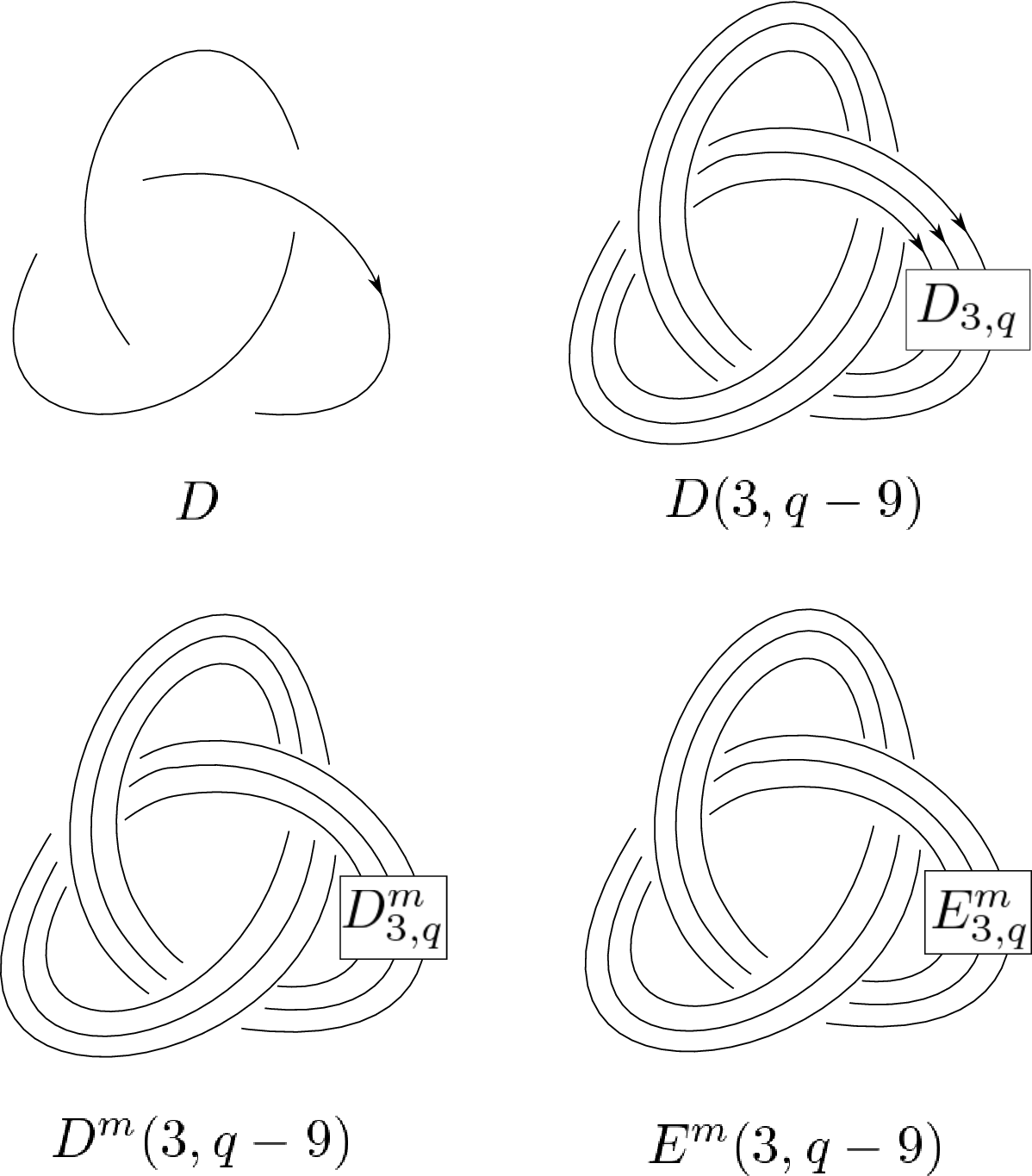}
\end{center}
\caption{Examples of $D(p, q)$, $D^{m}(p, q)$ and $E^{m}(p, q)$. }
\label{ex}
\end{figure}
\end{defn}
We first prove Lemma~$\ref{mainthm3}$, which implies Corollaries~$\ref{i_max1}$ and $\ref{i_max}$. 
\begin{lem}\label{mainthm3}
Let $K$ be an oriented knot and $D$ be a diagram of $K$ with $l_{+}$ positive crossings and $l_{-}$ negative crossings. 
Put $l=l_{+}+l_{-}$ and $f=l_{+}-l_{-}$. 
Then, for $n\geq l $ and any positive integer $k$, we have the following: 
\begin{align*}
\max\{i\in\mathbf{Z}|\KH^{i}(K(2k, 2k(n+f)))\neq 0\}=2k^{2}(n+f)
\end{align*}
and
\begin{align*}
2k(k+1)(n+f)&\leq \max\{i\in\mathbf{Z}|\KH^{i}(K(2k+1, (2k+1)(n+f)))\neq 0\}\\
&\leq 2k(k+1)(n+f)+l_{+}. 
\end{align*}
\end{lem}
We use Lemma~$\ref{lem_cable}$ below to prove Lemma~$\ref{mainthm3}$. 
Lemma~$\ref{lem_cable}$ gives upper bounds of $\max\{i\in\mathbf{Z}|\KH^{i}(K(p, p(n+f)))\neq 0\}$. 
\begin{lem}\label{lem_cable}
Let $k$ be a positive integer and $n\geq 0$. 
\begin{enumerate}
\item If $i>2k^{2}(n-l+1)+l(2k)^{2}$ and $n\geq l$, or $i>l(2k)^{2}$ and $n<l$, then we have $H^{i}(D^{m}(2k, 2k(n+f)+j))=0$ for any $j=1, \dots, 2k$ and $m=0, \dots, 2k-1$.
\item If $i>2k(k+1)(n-l+1)+l(2k+1)^{2}$ and $n\geq l$, or $i>l(2k+1)^{2}$ and $n<l$, then we have $H^{i}(D^{m}(2k+1, (2k+1)(n+f)+j)=0$ for any $j=1, \dots, 2k+1$ and $m=0, \dots, 2k$.
\end{enumerate}
\end{lem}
\begin{proof}[Proof of Lemma~$\ref{lem_cable}$ $(1)$]
We prove this by induction on $k$. 
For $k=1$, there is the following exact sequence: 
\begin{multline}
\cdots\to H^{i-1}(E^{1}(2, 2(n+f)+j))\to H^{i}(D(2, 2(n+f)+j)) \to\\ \to H^{i}(D(2, 2(n+f)+j-1))\to H^{i}(E^{1}(2, 2(n+f)+j))\to \cdots,  
\label{kanzenkeiretu}
\end{multline}
where $j=1, 2$ and $n\geq 0$. 
To study $H^{i}(D(2, 2(n+f)+j))$ and $H^{i}(D(2, 2(n+f)+j-1))$, we consider the diagram $E^{1}(2, 2(n+f)+j)$. 
\par
Note that for $j=1, 2$, the diagram $E^{1}(2, 2(n+f)+j)$ is a diagram of the unknot and has $2l+2n+j-1$ negative crossings. 
Hence for $i>2l+2n+j-1$ and $n\geq 0$, we have $H^{i}(E^{1}(2, 2(n+f)+j))=\KH^{i-(2l+2n+j-1)}(U)=0$. 
From the long exact sequence $(\ref{kanzenkeiretu})$, if $i>2l+2n+j$ and $n\geq 0$, then for $j=1, 2$ we obtain 
\begin{align*}
H^{i}(D(2, 2(n+f)+j))=H^{i}(D(2, 2(n+f)+j-1)). 
\end{align*}
By repeating the same process, if $i>2l+2n+j$ and $n\geq 0$, then for $j=1, 2$, we have 
\begin{align*}
H^{i}(D(2, 2(n+f)+j))&=H^{i}(D(2, 2(n+f)+j-1))\\
&=H^{i}(D(2, 2(n+f)+j-2))\\
&=\cdots =\\
&=H^{i}(D(2, 2f+1))\\
&=H^{i}(D(2, 2f)). 
\end{align*}
Since the diagram $D(2, 2f)$ has $4l$ crossings, we obtain $H^{i}(D(2, 2f))=0$ for any $i>4l$. 
Hence 
if $n\geq l$ and $i>2l+2n+j$, or $n<l$ and $i>4l$, then we obtain $H^{i}(D(2, 2(n+f)+j))=0$, where $j=1,2$. 
\par
Suppose that this lemma is true for $1, \dots, k-1$, that is, suppose that for $1\leq g<k$, $j=1, \dots, 2g$ and $m=0, \dots, 2g-1$, we have $H^{i}(D^{m}(2g, 2g(n+f)+j))=0$ if 
$i>2g^{2}(n-l+1)+l(2g)^{2}$ and $n\geq l$, or $i>l(2g)^{2}$ and $n<l$.  
\par
We will show that Lemma~$\ref{lem_cable}$ $(1)$ is true for $k$. We obtain the following exact sequence: 
\begin{multline}
\rightarrow H^{i-1}(E^{m}(2k, 2k(n+f)+j))\rightarrow H^{i}(D^{m-1}(2k, 2k(n+f)+j)) \label{kanzenkeiretu2}\\
\rightarrow H^{i}(D^{m}(2k, 2k(n+f)+j))\rightarrow H^{i}(E^{m}(2k, 2k(n+f)+j))\rightarrow ,  
\end{multline}
where $m=1, \dots, 2k-1$, $j=1, \dots, 2k$ and $n\geq 0$. 
To study $H^{i}(D^{m-1}(2k, 2k(n+f)+j))$ and $H^{i}(D^{m}(2k, 2k(n+f)+j))$, we use the following claim. 
\begin{claim}\label{key_claim}
Under the induction hypothesis in the proof of Lemma~$\ref{lem_cable}$ $(1)$, if $i>2k^{2}(n-l+1)+l(2k)^{2}-1$ and $n\geq l$, or $i>l(2k)^{2}-1$ and $n<l$, then 
we have $H^{i}(E^{m}(2k, 2k(n+f)+j))=0$ for any $j=1, \dots, 2k$ and $m=1, \dots, 2k-1$. 
\end{claim}
We will give a proof of Claim~$\ref{key_claim}$ in Section~$\ref{appendix}$. 
\par
From Claim~$\ref{key_claim}$ and the exact sequence $(\ref{kanzenkeiretu2})$, if $i>2k^{2}(n-l+1)+l(2k)^{2}$ and $n\geq l$, or $i>l(2k)^{2}$ and $n<l$, we have 
\begin{align*}
H^{i}(D^{m-1}(2k, 2k(n+f)+j))=H^{i}(D^{m}(2k, 2k(n+f)+j)) 
\end{align*}
for $m=1, \dots, 2k-1$ and $j=1, \dots, 2k$. 
\par
By repeating this process, if $i>2k^{2}(n-l+1)+l(2k)^{2}$ and $n\geq l$, or $i>l(2k)^{2}$ and $n<l$, for $m=0, \dots, 2k-1$ and $j=1, \dots, 2k$, we have 
\begin{align*}
H^{i}(D^{m}(2k, 2k(n+f)+j))&=H^{i}(D^{m+1}(2k, 2k(n+f)+j))\\
&=\cdots =\\
&=H^{i}(D^{2k-1}(2k, 2k(n+f)+j))\\
&=H^{i}(D^{0}(2k, 2k(n+f)+j-1))\\
&=H^{i}(D^{1}(2k, 2k(n+f)+j-1))\\
&=\cdots =\\
&=H^{i}(D^{2k-1}(2k, 2kf+1))\\
&=H^{i}(D(2k, 2kf))=0, 
\end{align*}
where the last equality follows from the fact that the diagram $D(2k, 2kf)$ has $l(2k)^{2}$ crossings. 
\end{proof}
%
%
\begin{proof}[Proof of Lemma~$\ref{lem_cable}$ $(2)$]
This proof is the same as the proof of Lemma~$\ref{lem_cable}$ $(1)$. 
We prove this by induction on $k$. 
For $k=1$, there is the following exact sequence: 
\begin{multline}
\cdots\to H^{i-1}(E^{m}(3, 3(n+f)+j))\to H^{i}(D^{m-1}(3, 3(n+f)+j)) \to\\ \to H^{i}(D^{m}(3, 3(n+f)+j))\to H^{i}(E^{m}(3, 3(n+f)+j))\to \cdots, 
\label{kanzenkeiretu3}
\end{multline}
where $m=1, 2$, $j=1,2,3$ and $n\geq 0$. 
\par
Note that  
\begin{itemize}
\item $E^{1}(3, 3(n+f)+1)$ is equivalent to $D$ and has $4n+5l_{-}+4l_{+}$ negative crossings, 
\item $E^{1}(3, 3(n+f)+2)$ is equivalent to $D$ and has $2+4n+5l_{-}+4l_{+}$ negative crossings, 
\item $E^{1}(3, 3(n+f)+3)$ is equivalent to $D\sqcup\bigcirc $ and has $3+4n+5l_{-}+4l_{+}$ negative crossings, 
\item $E^{2}(3, 3(n+f)+1)$ is equivalent to $D\sqcup\bigcirc $ and has $4n+5l_{-}+4l_{+}$ negative crossings, 
\item $E^{2}(3, 3(n+f)+2)$ is equivalent to $D$ and has $1+4n+5l_{-}+4l_{+}$ negative crossings, 
\item $E^{2}(3, 3(n+f)+3)$ is equivalent to $D$ and has $3+4n+5l_{-}+4l_{+}$ negative crossings. 
\end{itemize}
Hence $H^{i}(E^{m}(3, 3(n+f)+j))$ is isomorphic to $\KH^{i-n_{-}}(D)$ or $\KH^{i-n_{-}}(D\sqcup\bigcirc )$, where $n_{-}$ is the number of the negative crossings of $E^{m}(3, 3(n+f)+j)$. 
Since $D$ has only $l_{+}$ positive crossings, we have $\KH^{i-n_{-}}(D)=\KH^{i-n_{-}}(D\sqcup\bigcirc )=0$ if $i-n_{-}>l_{+}$. 
Hence $H^{i}(E^{m}(3, 3(n+f)+j))=0$ if $i>4n+3+5l$ and $n\geq 0$. 
\par
From the exact sequence $(\ref{kanzenkeiretu3})$, if $i>4n+4+5l$ and $n\geq 0$, we have 
\begin{align*}
H^{i}(D^{m}(3, 3(n+f)+j))=H^{i}(D^{m-1}(3, 3(n+f)+j)) 
\end{align*}
for $j=1,2,3$ and $m=1,2$. 
By repeating this process, if $n\geq l$ and $i>4n+4+5l$, or $n<l$ and $i>9l$, we obtain 
\begin{align*}
H^{i}(D^{m}(3, 3(n+f)+j))=H^{i}(D(3, 3f))=0, 
\end{align*}
for $j=1,2,3$ and $m=1,2$. 
\par
Suppose that this lemma is true for $1, \dots, k-1$, that is, suppose that for $1\leq g<k$, $j=1, \dots, 2g+1$ and $m=0, \dots, 2g$, we have $H^{i}(D^{m}(2g+1, (2g+1)(n+f)+j))=0$ if 
$i>2g(g+1)(n-l+1)+l(2g+1)^{2}$ and $n\geq l$, or $i>l(2g+1)^{2}$ and $n<l$. 
We will show that Lemma~$\ref{lem_cable}$ $(2)$ is true for $k$. 
We obtain the following exact sequence: 
\begin{multline}
\rightarrow H^{i-1}(E^{m}(2k+1, (2k+1)(n+f)+j))\rightarrow H^{i}(D^{m-1}(2k+1, (2k+1)(n+f)+j)) \label{kanzenkeiretu4}\\
\rightarrow H^{i}(D^{m}(2k+1, (2k+1)(n+f)+j))\rightarrow H^{i}(E^{m}(2k+1, (2k+1)(n+f)+j))\rightarrow ,  
\end{multline}
where $m=1, \dots, 2k$, $j=1, \dots, 2k+1$ and $n\geq 0$. 
To study $H^{i}(D^{m-1}(2k+1, (2k+1)(n+f)+j))$ and $H^{i}(D^{m}(2k+1, (2k+1)(n+f)+j))$, we use the following claim. 
\begin{claim}\label{key_claim2}
Under the induction hypothesis in the proof of Lemma~$\ref{lem_cable}$ $(2)$, if $i>2k(k+1)(n-l+1)+l(2k+1)^{2}-1$ and $n\geq l$, or $i>l(2k+1)^{2}-1$ and $n<l$ then 
we have $H^{i}(E^{m}(2k+1, (2k+1)(n+f)+j))=0$ for any $j=1, \dots, 2k+1$ and $m=1, \dots, 2k$. 
\end{claim}
We will give a proof of Claim~$\ref{key_claim2}$ in Section~$\ref{appendix}$. 
\par
From Claim~$\ref{key_claim2}$ and the exact sequence $(\ref{kanzenkeiretu4})$, if $i>2k(k+1)(n-l+1)+l(2k+1)^{2}$ and $n\geq l$, or $i>l(2k+1)^{2}$ and $n<l$, we have 
\begin{align*}
H^{i}(D^{m-1}(2k+1, (2k+1)(n+f)+j))=H^{i}(D^{m}(2k+1, (2k+1)(n+f)+j)) 
\end{align*}
 for $m=1, \dots, 2k$ and $j=1, \dots, 2k+1$. 
\par
By repeating this process, if $i>2k(k+1)(n-l+1)+l(2k+1)^{2}$ and $n\geq l$, or $i>l(2k+1)^{2}$ and $n<l$, then for $m=0, \dots, 2k$ and $j=1, \dots, 2k+1$, we obtain 
\begin{align*}
H^{i}(D^{m}(2k+1, (2k+1)(n+f)+j))=H^{i}(D(2k+1, (2k+1)f))=0. 
\end{align*}
\end{proof}
From Lemma~$\ref{lem_cable}$, we can prove Lemma~$\ref{mainthm3}$. 
\begin{proof}[Proof of Lemma~$\ref{mainthm3}$]
From Lemma~$\ref{lem_cable}$, we obtain 
\begin{align*}
\max\{i\in\mathbf{Z}|H^{i}(D(2k, 2k(n+f)))\neq 0\}\leq 2k^{2}(n+l). 
\end{align*}
Hence we have 
\begin{align*}
\max\{i\in\mathbf{Z}|\KH^{i}(K(2k, 2k(n+f)))\neq 0\}\leq 2k^{2}(n+l)-l_{-}(2k)^{2}=2k^{2}(n+f). 
\end{align*}
On the other hand, the dimension of $\operatorname{Lee}^{2k^{2}(n+f)}(K(2k, 2k(n+f)))$ is not zero. 
This implies that 
\begin{align*}
\max\{i\in\mathbf{Z}|\KH^{i}(K(2k, 2k(n+f)))\neq 0\}=2k^{2}(n+f). 
\end{align*}
Similarly we see that  
\begin{align*}
\max\{i\in\mathbf{Z}|\KH^{i}(K(2k+1, (2k+1)(n+f)))\leq  2k(k+1)(n+f)+l_{+} 
\end{align*}
and that the dimension of $\operatorname{Lee}^{2k(k+1)(n+f)}(K(2k+1, (2k+1)(n+f)))$ is not zero. 
Hence, we obtain  
\begin{align*}
2k(k+1)(n+f)&\leq \max\{i\in\mathbf{Z}|\KH^{i}(K(2k+1, (2k+1)(n+f)))\neq 0\}\\
&\leq 2k(k+1)(n+f)+l_{+}. 
\end{align*}
\end{proof}
We use Lemma~$\ref{lem2}$ below to prove Lemmas~$\ref{mainthm4}$ and $\ref{non-trivial2}$. 
\begin{lem}\label{lem2}
Let $K$ be a knot and $D$ be a knot diagram with $l_{+}$ positive crossings and $l_{-}$ negative crossings. 
Put $l=l_{+}+l_{-}$ and $f=l_{+}-l_{-}$. 
For any positive integer $k$ and any $n>l$, we have
\begin{align*}
\dim_{\mathbf{Q}} {\KH^{2k^{2}(n+f)}(K(2k, 2k(n+f)-1))}&=\dim_{\mathbf{Q}} {H^{2k^{2}(n+l)}(D(2k, 2k(n+f)-1))}\\
&=0. 
\end{align*} 
\end{lem}
\begin{proof}
We consider the following exact sequence: 
\begin{multline*}
\rightarrow H^{2k^{2}(n+l)-1}(E^{m}(2k, 2k(n+f-1)+j))\rightarrow H^{2k^{2}(n+l)}(D^{m-1}(2k, 2k(n+f-1)+j)) \\
\rightarrow H^{2k^{2}(n+l)}(D^{m}(2k, 2k(n+f-1)+j))\rightarrow H^{2k^{2}(n+l)}(E^{m}(2k, 2k(n+f-1)+j))\rightarrow , 
\end{multline*}
where $m=1, \dots, 2k-1$, $n\geq 0$ and $j=1, \dots, 2k-1$. 
We use the following claim 
to study $H^{2k^{2}(n+l)}(D^{m-1}(2k, 2k(n+f-1)+j))$ and $H^{2k^{2}(n+l)}(D^{m}(2k, 2k(n+f-1)+j))$. 
\begin{claim}\label{key_claim3}
If $i>l(2k)^{2}+2k^{2}(n-l)-2$ and $n>l$ we have 
$H^{i}(E^{m}(2k, 2k(n+f-1)+j))=0$ for any $m=1, \dots, 2k-1$ and $j=1, \dots, 2k-1$ . 
\end{claim}
Compare Claim~$\ref{key_claim3}$ to Claim~$\ref{key_claim}$ (the main differences are the ranges of $i$ and $j$). 
We will give a proof of Claim~$\ref{key_claim3}$ in Section~$\ref{appendix}$. 
\par
From Claim~$\ref{key_claim3}$ and the above exact sequence, 
if $i>l(2k)^{2}+2k^{2}(n-l)-1$ and $n>l$, we have 
\begin{align*}
H^{i}(D^{m-1}(2k, 2k(n+f-1)+j))=H^{i}(D^{m}(2k, 2k(n+f-1)+j)), 
\end{align*}
where $m=1, \dots, 2k-1$ and $j=1, \dots, 2k-1$. 
In particular, 
if $i=2k^{2}(n+l)$, $m=1$ and $j=2k-1$, we obtain 
\begin{align*}
H^{2k^{2}(n+l)}(D(2k, 2k(n+f)-1))&=H^{2k^{2}(n+l)}(D^{0}(2k, 2k(n+f-1)+2k-1))\\
&=H^{2k^{2}(n+l)}(D^{1}(2k, 2k(n+f-1)+2k-1)). 
\end{align*}
By repeating this process, we have 
\begin{align*}
H^{2k^{2}(n+l)}(D(2k, 2k(n+f)-1))&=H^{2k^{2}(n+l)}(D^{1}(2k, 2k(n+f-1)+2k-1))\\
&=H^{2k^{2}(n+l)}(D^{2}(2k, 2k(n+f-1)+2k-1))\\
&=\cdots= \\
&=H^{2k^{2}(n+l)}(D^{2k-1}(2k, 2k(n+f-1)+2k-1))\\
&=H^{2k^{2}(n+l)}(D^{0}(2k, 2k(n+f-1)+2k-2))\\
&=\cdots= \\
&=H^{2k^{2}(n+l)}(D(2k, 2k(n+f-1)))=0, \\ 
\end{align*}
where the last equality follows from Lemma~$\ref{mainthm3}$. 
\end{proof}
By using Lemma~$\ref{lem2}$, we will prove Lemmas~$\ref{mainthm4}$ and $\ref{non-trivial2}$. 
Lemma~$\ref{mainthm4}$ is an extension of Lemma~$\ref{thm2}$. 
\begin{lem}\label{mainthm4}
Let $K$ be a knot and $D$ be a diagram of $K$ with $l_{+}$ positive crossings and $l_{-}$ negative crossings. 
Put $l=l_{+}+l_{-}$ and $f=l_{+}-l_{-}$. 
Then for any positive integer $k$ and any $n>l$, we have
\begin{align*}
\dim_{\mathbf{Q}} \KH^{2k^{2}(n+f)}(K(2k, 2k(n+f)))=\begin{pmatrix}
2k\\
k
\end{pmatrix}.
\end{align*}
\end{lem}
\begin{proof}
As in the proof of Lemma~$\ref{mainthm2}$, in order to prove this lemma, it is sufficient to prove the following: 
\begin{align}
\dim_{\mathbf{Q}} H^{2k^{2}(n+l)}(D^{i}(2k, 2k(n+f)))=2
\begin{pmatrix}
2k-1-i\\
k
\end{pmatrix}. \label{b}
\end{align}
where $0\leq i\leq 2k-1$ (for convenience, we define 
$\begin{pmatrix}
a\\
b
\end{pmatrix}=0$ if $0\leq a<b$). 
To prove $(\ref{b})$, we use induction on $k$. 
\par
For $k=1$, 
from Lemma~$\ref{lem2}$ we obtain 
\begin{center}
$\dim_{\mathbf{Q}} H^{2k^{2}(n+l)}(D^{1}(2, 2(n+f)))=\dim_{\mathbf{Q}} H^{2k^{2}(n+l)}(D(2, 2(n+f)-1))=0$. 
\end{center}
Hence we have the following exact sequence: 
\begin{align*}
\dots\rightarrow H^{2(n+l)-1, j-1}(E^{1}(2, 2(n+f)))\rightarrow H^{2(n+l), j}(D(2, 2(n+f)))\rightarrow 0. 
\end{align*}
From the above exact sequence, we obtain 
\begin{align*}
\sum_{j}\dim_{\mathbf{Q}} H^{2(n+l), j}(D(2, 2(n+f)))&\leq \sum_{j} \dim_{\mathbf{Q}} H^{2(n+l)-1, j-1}(E^{1}(2, 2(n+f))). 
\end{align*}
Since the diagram $E^{1}(2, 2(n+f))$ is equivalent to a diagram of the unknot and has $2(n+l)-1$ negative crossings, we have  
\begin{align*}
\sum_{j} \dim_{\mathbf{Q}} H^{2(n+l)-1, j-1}(E^{1}(2, 2(n+f)))
&=\sum_{j}\dim_{\mathbf{Q}} \KH^{0, j}(U)
=2,  
\end{align*}
where $U$ is the unknot. 
Hence we obtain 
\begin{align*}
\sum_{j}\dim_{\mathbf{Q}} H^{2(n+l), j}(D(2, 2(n+f)))\leq 2.  
\end{align*}
On the other hand, the dimension of $\operatorname{Lee}^{2(n+f)}(D(2, 2(n+f)))$ is $2$. 
Hence we obtain 
\begin{align*}
\dim_{\mathbf{Q}} H^{2(n+l)}(D(2, 2(n+f)))=2. 
\end{align*}
\par
Suppose that $(\ref{b})$ is true for $1, \dots, k-1$, that is, suppose that for $1\leq h< k$, $n>0$ and $i=0, \dots, 2h-1$ we have 
\begin{align}
\dim_{\mathbf{Q}} H^{2h^{2}(n+l)}(D^{i}(2h, 2h(n+f)))=2
\begin{pmatrix}
2h-1-i\\
h
\end{pmatrix}. \label{b_suppose}
\end{align}
We will show that $(\ref{b})$ is true for $k$. 
We have the following long exact sequence: 
\begin{multline}
\cdots\to H^{2k^{2}(n+l)-1, j-1}(E^{i+1}(2k, 2k(n+f)))\xrightarrow {g^{i}_{j}} \\
H^{2k^{2}(n+l), j}(D^{i}(2k, 2k(n+f)))\xrightarrow{f^{i}_{j}} H^{2k^{2}(n+l), j}(D^{i+1}(2k, 2k(n+f)))\to \cdots . \label{kanzen4}
\end{multline}
From the exact sequence $(\ref{kanzen4})$ and the same discussion in $(\ref{add1})$, we obtain 
\begin{align}
&\sum_{j}\dim_{\mathbf{Q}} H^{2k^{2}(n+l), j}(D^{i}(2k, 2k(n+f)))\label{add5}\\
&\leq \sum_{j}\sum_{m=i+1}^{2k-1}\dim_{\mathbf{Q}} H^{2k^{2}(n+l)-1, j-1}(E^{m}(2k, 2k(n+f)))\nonumber\\
&\ \ \ \ \ \ \ \ \ \ \ \ \ \ \ \ \ \ \ \ \ \ +\dim_{\mathbf{Q}} H^{2k^{2}(n+l)}(D(2k, 2k(n+f)-1)). \nonumber
\end{align}
From Lemma~$\ref{lem2}$, we have $\dim_{\mathbf{Q}} H^{2k^{2}(n+l)}(D(2k, 2k(n+f)-1))=0$. 
To compute $\dim_{\mathbf{Q}} H^{2k^{2}(n+l)-1}(E^{m}(2k, 2k(n+f)))$, we consider $E^{m}(2k, 2k(n+f))$. 
\par
Note that $E^{m}(2k, 2k(n+f))$ is equivalent to the diagram $D^{m-2}(2k-2, (2k-2)(n+f))$ for $m\geq 2$. 
We give $E^{m}(2k, 2k(n+f))$ an orientation such that all crossings of $D^{m-2}(2k-2, (2k-2)(n+f))$ are positive. 
Then $E^{m}(2k, 2k(n+f))$ has $4kn-2n-1+2(2k-1)l_{+}+((2k)^{2}-2(2k-1))l_{-}$ negative crossings, 
where $l_{+}$ and $l_{-}$ are the number of the positive and negative crossings of $D$, respectively. 
Hence for $m\geq 2$ we obtain
\begin{align}
&\dim_{\mathbf{Q}} H^{2k^{2}(n+l)-1}(E^{m}(2k, 2k(n+f)))\label{add6}\\
&=\dim_{\mathbf{Q}} H^{2(k-1)^{2}(n+l)}(D^{m-2}(2k-2, (2k-2)(n+f)))
=2
\begin{pmatrix}
2k-1-m\\
k-1
\end{pmatrix}. \nonumber 
\end{align}
Similarly, $E^{1}(2k, 2k(n+f))$ is equivalent to the diagram $D(2k-2, (2k-2)(n+f))\sqcup \bigcirc$, where $\bigcirc$ is a circle in the plane. 
We give $E^{1}(2k, 2k(n+f))$ an orientation such that all crossings of $D(2k-2, (2k-2)(n+f))\sqcup \bigcirc$ are positive. 
Then $E^{1}(2k, 2k(n+f))$ has $4kn-2n-1+2(2k-1)l_{+}+((2k)^{2}-2(2k-1))l_{-}$ negative crossings. 
%
%
Hence we obtain 
\begin{align}
&\dim_{\mathbf{Q}} H^{2k^{2}(n+l)-1}(E^{m}(2k, 2k(n+f)))\label{add7}\\
&=\dim_{\mathbf{Q}} H^{2(k-1)^{2}(n+l)}(D^{m-2}(2k-2, (2k-2)(n+f))\sqcup\bigcirc )
=2
\begin{pmatrix}
2k-2\\
k-1
\end{pmatrix}.  \nonumber 
\end{align}
From $(\ref{add5})$, $(\ref{add6})$ and $(\ref{add7})$, we have 
\begin{align}
&\sum_{j}\dim_{\mathbf{Q}} H^{2k^{2}(n+l), j}(D^{i}(2k, 2k(n+f)))\label{hutousiki2}\\
&\leq\sum_{m=i+1}^{2k-1}2\begin{pmatrix}
2k-1-m\\
k-1
\end{pmatrix}
=2\begin{pmatrix}
2k-1-i\\
k
\end{pmatrix}. \nonumber
\end{align}
\par
Finally we will prove that the inequality in $(\ref{hutousiki2})$ is in fact an equality. 
%
At first, we consider the case where $i=0$. 
The dimension of $\operatorname{Lee}^{2k^{2}(n+f)}(D(2k, 2k(n+f)))$ is $\begin{pmatrix}
2k\\
k
\end{pmatrix}$. 
Hence, we have 
\begin{align*} 
\begin{pmatrix}
2k\\
k
\end{pmatrix}&=\dim_{\mathbf{Q}} \operatorname{Lee}^{2k^{2}(n+f)}(D(2k, 2k(n+f)))\\
&\leq \dim_{\mathbf{Q}} H^{2k^{2}(n+l)}(D(2k, 2k(n+f)))
\leq \begin{pmatrix}
2k\\
k
\end{pmatrix}. 
\end{align*}
This implies that we have the equality in $(\ref{hutousiki2})$ for $i=0$. 
This fact implies that for any $j\in \mathbf{Z}$ and $m=0, \dots, 2k-2$, the maps $g^{m}_{j}$ and $f^{m}_{j}$ in $(\ref{kanzen4})$ are injective and surjective, respectively. 
Hence, we have the equality in $(\ref{hutousiki2})$ for $i=0, \dots, 2k-1$ and we obtain
\begin{align*}
\dim_{\mathbf{Q}} H^{2k^{2}(n+l)}(D^{i}(2k, 2k(n+f)))
&=\sum_{j}\dim_{\mathbf{Q}} H^{2k^{2}(n+l), j}(D^{i}(2k, 2k(n+f)))\\
&=2\begin{pmatrix}
2k-1-i\\
k
\end{pmatrix}.
\end{align*}
\end{proof}
%
Next we prove Lemma~$\ref{non-trivial2}$. 
\begin{lem}\label{non-trivial2}
Let $K$ be a knot and $D$ be a diagram of $K$ with $l_{+}$ positive crossings and $l_{-}$ negative crossings. 
Put $l=l_{+}+l_{-}$ and $f=l_{+}-l_{-}$. 
Then for any $n>l$, any positive integer $k$ and $i=0, \dots, k$, we have 
\begin{align*}
\KH^{2k^{2}(n+f), 6k^{2}(n+f)-2i}(K(2k, 2k(n+f)))\neq 0. 
\end{align*}
\end{lem}
\begin{proof}
We use induction on $k$. In the case where $k=1$, we need to prove 
\begin{align*}
\KH^{2(n+f), 6(n+f)-1\pm 1}(D(2, 2(n+f)))\neq 0. 
\end{align*}
We have the exact sequence 
\begin{multline*}
\rightarrow H^{2(n+l)-1, j-1}(E^{1}(2, 2(n+f)))\rightarrow H^{2(n+l), j}(D(2, 2(n+f)))\\
\rightarrow H^{2(n+l), j}(D^{1}(2, 2(n+f)))\rightarrow . 
\end{multline*}
It follows from Lemma~$\ref{lem2}$ that 
\begin{align*}
H^{2(n+l), j}(D^{1}(2, 2(n+f)))=H^{2(n+l), j}(D(2, 2(n+f)-1))=0. 
\end{align*}
The diagram $E^{1}(2, 2(n+f))$ is equivalent to a diagram of the unknot and has $2l+2n-1$ negative crossings and $2l$ positive crossings. 
Hence we have 
\begin{align*}
H^{2(n+l)-1, j-1}(E^{1}(2, 2(n+f)))=
\begin{cases}
\mathbf{Q} & \text{if } j=2l+4n-1\pm 1, \\
0 &\text{otherwise}. 
\end{cases}
\end{align*}
By Lemma~$\ref{mainthm4}$, we have $\dim_{\mathbf{Q}} H^{2(n+l)}(D(2, 2(n+f)))=2$.  
From the above exact sequence, we have $H^{2(n+l)-1, j-1}(E^{1}(2, 2(n+f)))=H^{2(n+l), j}(D(2, 2(n+f)))$ 
since $\dim_{\mathbf{Q}} H^{2(n+l)}(D(2, 2(n+f)))=2=\dim_{\mathbf{Q}} H^{2(n+l)-1}(E^{1}(2, 2(n+f)))$. 
Hence we obtain 
\begin{align*}
\KH^{2(n+f), 6(n+f)-1\pm 1}(D(2, 2(n+f)))
&=H^{2(n+l), 2l+4n-1\pm 1}(D(2, 2(n+f)))\\
&=H^{2(n+l)-1, 2l+4n-2\pm 1}(E^{1}(2, 2(n+f)))\\
&=\mathbf{Q}. 
\end{align*}
\par
Suppose that Lemma~$\ref{non-trivial2}$ is true for $1, \dots, k-1$, that is, suppose that for $1\leq h<k$, $n>0$ and $i=0, \dots, h$, we have 
\begin{align}
\KH^{2h^{2}(n+f), 6h^{2}(n+f)-2i}(K(2h, 2h(n+f)))\neq 0. \label{nontri2_suppose}
\end{align}
From the proof of Lemma~$\ref{mainthm4}$ (, recall that the inequality $(\ref{add5})$ is in fact an equality), we have 
\begin{align}
&\dim_{\mathbf{Q}} H^{2k^{2}(n+l), j}(D(2k, 2k(n+f))) \label{add8}\\
&\geq \dim_{\mathbf{Q}} H^{2k^{2}(n+l)-1,j-1}(E^{1}(2k, 2k(n+f))) \nonumber \\
&\ \ \ \ +\dim_{\mathbf{Q}} H^{2k^{2}(n+l)-1,j-1}(E^{2}(2k, 2k(n+f))). \nonumber
\end{align}
The diagram $E^{1}(2k, 2k(n+f))$ is equivalent to $D(2k-2, (2k-2)(n+f))\sqcup\bigcirc $, where $\bigcirc $ is a circle in the plane. 
We give $E^{1}(2k, 2k(n+f))$ an orientation such that all crossings of $D(2k-2, (2k-2)(n+f))\sqcup\bigcirc $ are positive. 
Then $E^{1}(2k, 2k(n+f))$ has $2(2k-1)(f+n)-1+l_{-}(2k)^{2}$ negative crossings and $(2k)^{2}l+(2k-1)2kn-1$ crossings. 
Similarly, the diagram $E^{2}(2k, 2k(n+f))$ is equivalent to $D(2k-2, (2k-2)(n+f))$. 
We give $E^{2}(2k, 2k(n+f))$ an orientation such that all crossings of $D(2k-2, (2k-2)(n+f))$ are positive. 
Then $E^{2}(2k, 2k(n+f))$ has $2(2k-1)(f+n)-1+l_{-}(2k)^{2}$ negative crossings and $(2k)^{2}l+(2k-1)2kn-2$ crossings. 
From $(\ref{add8})$, we have 
\begin{align*}
&\dim_{\mathbf{Q}} \KH^{2k^{2}(n+f), 6k^{2}(n+f)-2i}(D(2k, 2k(n+f)))\\
&\geq \dim_{\mathbf{Q}} \KH^{2(k-1)^{2}(n+f), 6(k-1)^{2}(n+f)-2i+1}(D(2k-2, (2k-2)(n+f))\sqcup\bigcirc )\\
&\ \ \ \ +\dim_{\mathbf{Q}} \KH^{2(k-1)^{2}(n+f), 6(k-1)^{2}(n+f)-2i}(D(2k-2, (2k-2)(n+f))). 
\end{align*}
By the induction hypothesis $(\ref{nontri2_suppose})$, the first term of the last expression is not zero for $i=1, \dots, k$, and the second term is not zero for $i=0, \dots, k-1$. 
This completes this proof. 
\end{proof}
%
%
\begin{rem}
In general Lemma~$\ref{lem2}$ is not true for $(2k+1, (2k+1)n)$-cable links, that is, $\dim_{\mathbf{Q}} \KH^{2k(k+1)(n+f)}(D(2k+1, (2k+1)(n+f)-1))\neq 0$ even though $n>l$. 
A reason is that the maximal homological degree of the Khovanov homology of a $(2k+1, (2k+1)n)$-cable link is not equal to that of the Lee homology of the link. 
Since we need Lemma~$\ref{lem2}$ to prove Lemmas~$\ref{mainthm4}$ and $\ref{non-trivial2}$, we cannot obtain results for  $(2k+1, (2k+1)n)$-cable links corresponding to these lemmas by the same methods. 
\end{rem}
From Lemma~$\ref{non-trivial2}$, we obtain the following. 
\begin{cor}\label{cor_thick}
Let $K$ be a positive knot and $D$ be a positive diagram of $K$ with $l$ crossings. 
Then for any $n>l$ and any positive integer $k$, the homological thickness 
$\operatorname{hw}(K(2k, 2k(n+l)))$ is greater than or equal to $k(k-1)(n+l)+2+ks(K)$, where $s(K)$ is the Rasmussen invariant of $K$. 
\end{cor}
\begin{proof}
By Lemma~$\ref{non-trivial2}$, we have 
\begin{center}
$\KH^{2k^{2}(n+l), 6k^{2}(n+l)-2k}(K(2k, 2k(n+l)))\neq 0$. \par
\end{center}
Since $D(2k, 2k(n+l))$ is also positive diagram, from Theorem~$\ref{khovanov2}$, we obtain 
\begin{center}
$\KH^{0, 4k^{2}l+2kn(2k-1)-2ks_{0}(D)+2}(K(2k, 2k(n+l)))\neq 0$,  \par
\end{center}
where $s_{0}(D)$ is the number of Seifert circles of $D$. 
Hence 
\begin{align*}
\operatorname{hw}(K(2k, 2k(n+l)))&\geq k(k-1)(n+l)+2+k(l+1-s_{0}(D)). 
\end{align*}
It is known that the Rasmussen invariant $s(K)$ of a positive knot $K$ is $l+1-s_{0}(D)$, where $D$ is a positive diagram of $K$ with $l$ crossings (see \cite[Section~$5.2$]{rasmussen1}). 
Hence we obtain 
\begin{align*}
\operatorname{hw}(K(2k, 2k(n+l)))&\geq k(k-1)(n+l)+2+k\cdot s(K). 
\end{align*}
\end{proof}
\begin{rem}
Corollary~$\ref{cor_thick}$ is an extension of Corollary~$\ref{torus_thick}$. 
From Theorem~$\ref{dalt_hw}$, if $n$ is sufficiently large, the $(2k, 2kn)$-cabling of any positive knot $K$ has no diagram which is alternating after $k(k-1)n+ks(K)-1$ or less crossing changes. 
\end{rem}
\section{An application for twisted Whitehead doubles}\label{apply}
In this section, we consider twisted Whitehead doubles of any knot and compute their Khovanov homologies. \par
Let $K$ be a knot. 
A twisted Whitehead double of $K$ is represented by the diagram $L(D, q)$ in Figure~$\ref{double}$, where $D$ is a diagram of $K$ and $q$ is an integer. 
The right picture in Figure~$\ref{double1}$ is a twisted Whitehead double of the left-handed trefoil. 
\par
A cable link is obtained from a twisted Whitehead double of any knot by smoothing at a crossing. 
In Section~$\ref{main2}$, we give some computations of the Khovanov homology groups of cable links. 
By applying these computations, we will calculate the Khovanov homology groups of a twisted Whitehead double of any knot with sufficiently many twists. 
Moreover we compute their Rasmussen invariants (Corollary~$\ref{ras_double}$). \par
Let $D$ be a knot diagram with $l_{+}(D)$ positive crossings and $l_{-}(D)$ negative crossings. 
Put $l=l_{+}(D)+l_{-}(D)$ and $f=l_{+}(D)-l_{-}(D)$. 
Let $L(D, q)=L$, $L_{0}$ and $L_{1}$ be knot diagrams depicted in Figure~$\ref{double}$, where $q$ is a non-negative integer (for example, see Figure~$\ref{double1}$). 
In the case where $q$ is negative, we define $L(D, q)$ as the mirror image of $L(-D, -q+1)$, where $-D$ is the mirror image of $D$. 
\begin{figure}[!h]
\begin{center}
\includegraphics[scale=0.4]{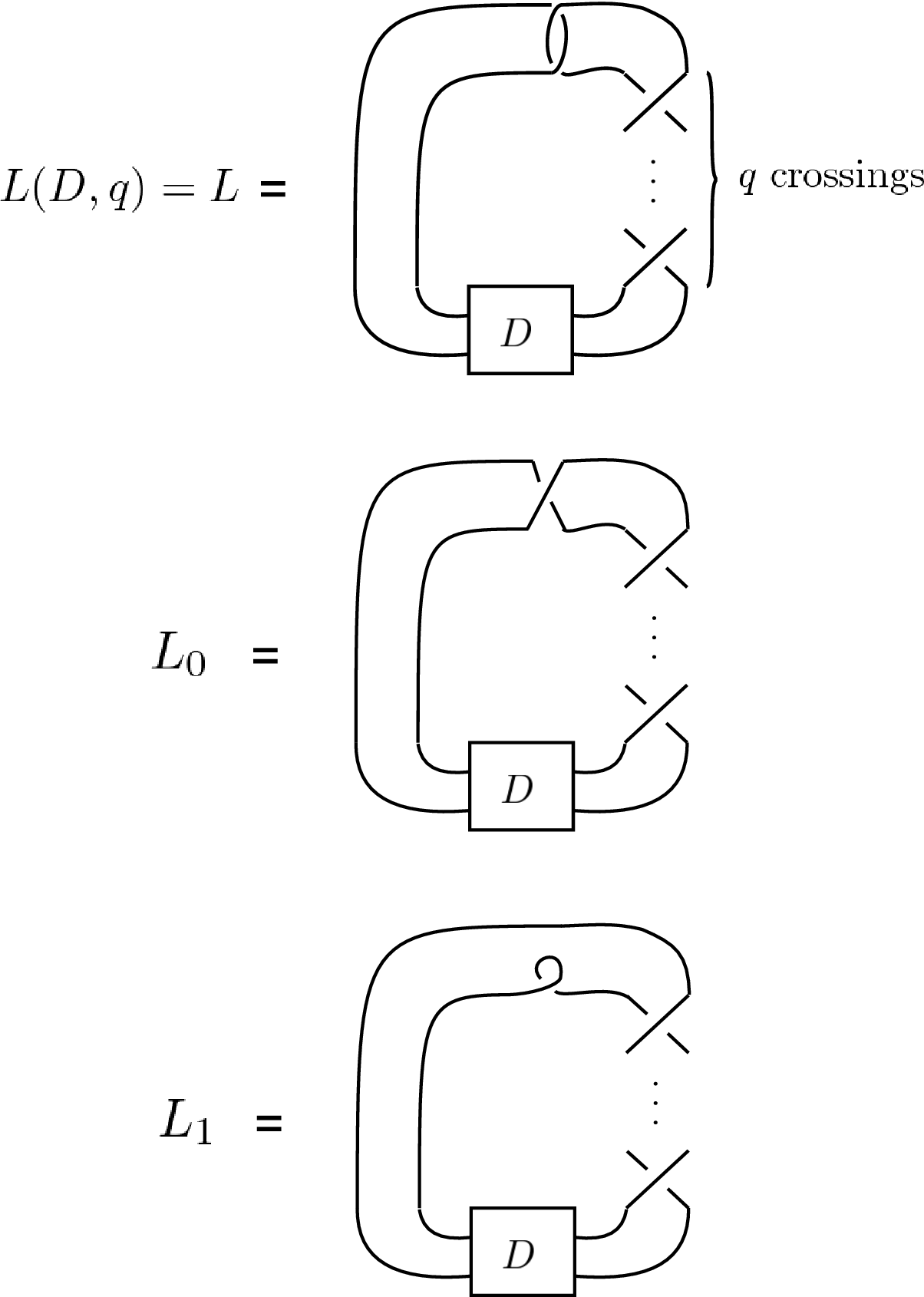}
\end{center}
\caption{$L(D, q)=L$, $L_{0}$ and $L_{1}$, where $q$ is non-negative. }
\label{double}
\end{figure}
\begin{figure}[!h]
\begin{center}
\includegraphics[scale=0.4]{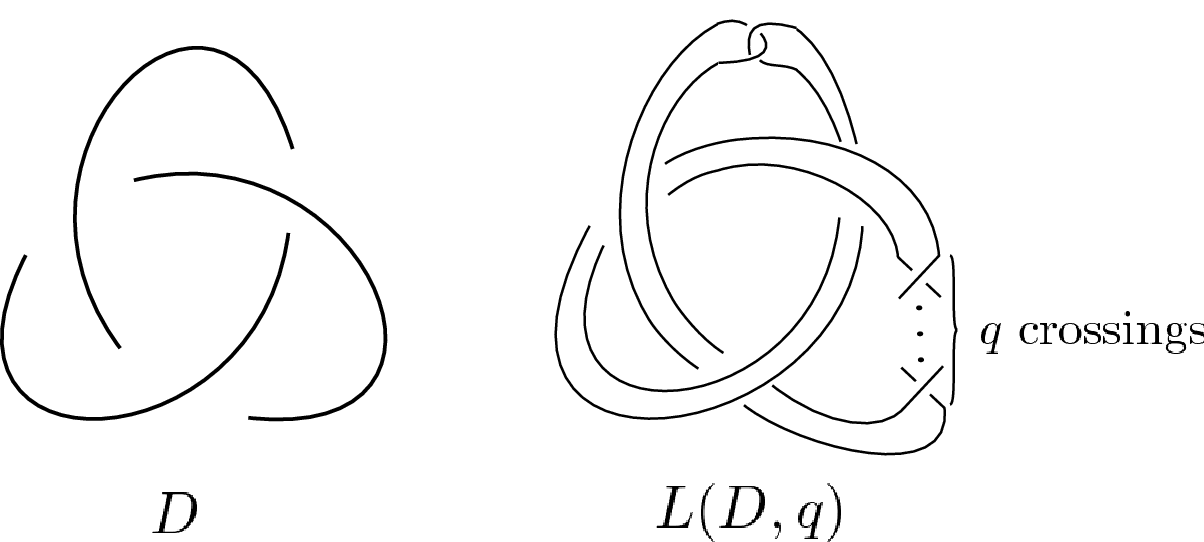}
\end{center}
\caption{An example of $L(D, q)$. }
\label{double1}
\end{figure}
\par
By the definition, we have 
\begin{align*}
H^{i, j}(L_{1})&=H^{i-1,j-2}(D(2, q+2f)), \\
H^{i, j}(L_{0})&=H^{i-1,j-1}(D(2, q-1+2f)). 
\end{align*}
To study the Khovanov homology of $L(D,q)$, we compute $H^{i, j}(D(2, q-1+2f))$ for some $i$ and $j$. 
\begin{lem}\label{lem3}
For $n>l+1$, we have
\begin{align*}
H^{2(n+l)-1, j}(D(2, 2(n+f)-1))=
\begin{cases}
\mathbf{Q}& \text{if }j=2l+4n-2, \\
0&\text{if }j\neq 2l+4n-3\pm 1, 
\end{cases}
\end{align*}
and for $n>l$ and any $i\geq 2(n+l)$, we have 
\begin{align*}
H^{i}(D(2, 2(n+f)-1))=0. 
\end{align*}
\end{lem}
\begin{proof}
We obtain the following exact sequence: 
\begin{multline*}
\rightarrow H^{2(n+l)-2, j}(D^{1}(2, 2(n+f)-1))\rightarrow H^{2(n+l)-2, j-1}(E^{1}(2, 2(n+f)-1))\rightarrow \\
H^{2(n+l)-1, j}(D(2, 2(n+f)-1))\rightarrow H^{2(n+l)-1, j}(D^{1}(2, 2(n+f)-1))\rightarrow , 
\end{multline*}
where $E^{m}(p, q)$ and $D^{m}(p,q)$ are given in Figure~$\ref{E(mpq)}$. By Lemma~$\ref{mainthm3}$ we have
\begin{align*}
H^{2(n+l)-1, j}(D^{1}(2, 2(n+f)-1))&=H^{2(n+l)-1, j}(D(2, 2(n+f)-2))=0. 
\end{align*}
The diagram $E^{1}(2, 2(n+f)-1)$ is a diagram of the unknot and has $2l+2n-2$ negative crossings and $2l$ positive crossings. Hence we have
\begin{align*}
H^{2(n+l)-2, j-1}(E^{1}(2, 2(n+f)-1))=
\begin{cases}
\mathbf{Q}& \text{if }j=2l+4n-3\pm 1, \\
0& otherwise. 
\end{cases}
\end{align*}
By Lemmas~$\ref{non-trivial2}$ and $\ref{mainthm4}$, we obtain 
\begin{align*}
H^{2(n+l)-2, j}(D(2, 2(n+f)-2))=
\begin{cases}
\mathbf{Q}& \text{if }j=2l+4n-5\pm 1, \\
0& otherwise. 
\end{cases}
\end{align*}
From the above exact sequence, we have 
\begin{align*}
H^{2(n+l)-1, j}(D(2, 2(n+f)-1))=
\begin{cases}
\mathbf{Q}& \text{if }j=2l+4n-2, \\
0&\text{if }j\neq 2l+4n-3\pm 1. 
\end{cases}
\end{align*}
The second claim follows from Lemmas~$\ref{lem2}$ and $\ref{mainthm3}$. 
\end{proof}
By using Lemma~$\ref{lem3}$, we can compute some Khovanov homology groups of $L(D, q)$. 
\begin{prop}\label{lem4}
Let $D$ be a knot diagram with $l_{+}(D)$ positive crossings and $l_{-}(D)$ negative crossings. 
Put $l=l_{+}(D)+l_{-}(D)$. 
Let $n$ be an integer which is greater than $l$. 
\par
(I) In the case where $q=2n$, we have
\begin{align*}
\KH^{0, j}(L(D, q))=
\begin{cases}
\mathbf{Q}& \text{if }j=-2\pm 1, \\
0&\text{otherwise}. 
\end{cases}
\end{align*}
\par
(II) In the case where $q=2n+1$, we have
\begin{align*}
\KH^{2, j}(L(D, q))=
\begin{cases}
\mathbf{Q}& \text{if }j=5, \\
0&\text{if }j\neq 5, 3. 
\end{cases}
\end{align*}
\end{prop}
\begin{proof}
Put $f=l_{+}(D)-l_{-}(D)$. \par
(I) In the case where $q=2n$. \par
From Lemma~$\ref{mainthm4}$, we obtain $\dim_{\mathbf{Q}} H^{2(n+l)}(D(2, 2(f+n)))=2$.  
From Lemma~$\ref{non-trivial2}$, we have 
$H^{2(n+l), 4n+2l-1\pm 1}(D(2, 2(f+n)))\neq 0$. Hence we obtain 
\begin{align*}
H^{2(n+l)+1, j}(L_{1})=H^{2(n+l), j-2}(D(2, 2(f+n)))
=
\begin{cases}
\mathbf{Q}& \text{if }j=4n+2l+1\pm 1, \\
0 &\text{otherwise}. 
\end{cases}
\end{align*}
From Lemma~$\ref{lem3}$, we obtain 
$H^{i, j}(L_{0})=H^{i-1, j-1}(D(2, 2(f+n)-1))=0$ if $i>2(n+l)$. 
Now there is the following exact sequence. 
\begin{center}
$\rightarrow H^{2(n+l)+1, j}(L_{0})\rightarrow H^{2(n+l)+1, j-1}(L_{1})\rightarrow H^{2(n+l)+2, j}(L)\rightarrow H^{2(n+l)+2, j}(L_{0})\rightarrow $. 
\end{center}
Since $H^{2(n+l)+1, j}(L_{0})=H^{2(n+l)+2, j}(L_{0})=0$, we have 
\begin{align*}
H^{2(n+l)+2, j}(L)&=
\begin{cases}
\mathbf{Q}& \text{if }j=4n+2l+2\pm 1,\\ 
0 &\text{otherwise }. 
\end{cases}
\end{align*}
The diagram $L=L(D, 2n)$ has $2n+2+2l$ negative crossings and $2l$ positive crossings. By the definition, we obtain 
\begin{align*}
\KH^{0, j}(L(D, q))=
\begin{cases}
\mathbf{Q}& \text{if }j=-2\pm 1, \\
0&\text{otherwise}. 
\end{cases}
\end{align*}
\par
(II) In the case where $q=2n+1$. \par
We can proof this by the same method as (I). 
It follows from Lemmas~$\ref{mainthm3}$ and $\ref{lem3}$ that
\begin{align*}
H^{2(n+l)+2, j}(L_{1})=H^{2(n+l)+1, j-2}(D(2, 2f+2n+1))
=
\begin{cases}
\mathbf{Q}& \text{if }j=4n+2l+4,\\ 
0 &\text{if }j\neq 4n+2l+3\pm 1, 
\end{cases}
\end{align*}
and 
$H^{i, j}(L_{0})=H^{i-1, j-1}(D(2, 2f+2n))=0$ if $i>2(n+l)+1$. 
Now we have the following exact sequence: 
\begin{center}
$H^{2(n+l)+2, j}(L_{0})\rightarrow H^{2(n+l)+2, j-1}(L_{1})\rightarrow H^{2(n+l)+3, j}(L)\rightarrow H^{2(n+l)+3, j}(L_{0})$. 
\end{center}
Since $H^{2(n+l)+2, j}(L_{0})=H^{2(n+l)+3, j}(L_{0})=0$, we obtain 
\begin{align*}
H^{2(n+l)+3, j}(L)&=
H^{2(n+l)+2, j-1}(L_{1})
=
\begin{cases}
\mathbf{Q}& \text{if }j=4n+2l+5,\\ 
0 &\text{if }j\neq 4n+2l+4\pm 1. 
\end{cases}
\end{align*}
The diagram $L=L(D, 2n+1)$ has $2n+1+2l$ negative crossings and $2+2l$ positive crossings. By the definition we have
\begin{align*}
\KH^{2, j}(L(D, q))=
\begin{cases}
\mathbf{Q}& \text{if }j=5, \\
0&\text{if }j\neq 5, 3. 
\end{cases}
\end{align*}
\end{proof}
\begin{cor}\label{cor5-1}
Let $D$ be a knot diagram with $l_{+}(D)$ positive crossings and $l_{-}(D)$ negative crossings. 
Put $l=l_{+}(D)+l_{-}(D)$. 
Let $n$ be an integer which is greater than $l$. 
Then we have $s(L(D, 2n))=-2$, where $s(K)$ is the Rasmussen invariant of a knot $K$. 
\end{cor}
\begin{proof}
From Theorem~$\ref{4.3}$, we have $\dim_{\mathbf{Q}}\operatorname{Lee}^{0}(L(D, 2n))=2$. 
Let $s_{\max}$ and $s_{\min}$ be its generators. 
Assume that the $q$-grading of $s_{\max}$ is greater than that of $s_{\min}$. 
From the definition of the Rasmussen invariant, the $q$-grading of $s_{\max}$ is $s(L(D, 2n))+1$ and that of $s_{\min}$ is $s(L(D, 2n))-1$. 
Since there is a spectral sequence whose $E_{\infty}$-page is the Lee homology and $E_{2}$-page is the Khovanov homology, we have
\begin{align*}
\KH^{0, s(L(D, 2n))\pm 1}(L(D, 2n))\neq 0. 
\end{align*}
From Proposition~$\ref{lem4}$ (I), we have $s(L(D, 2n))=-2$. 
\end{proof}
In \cite{concordance_doubled} Livingston and Naik showed Theorem~$\ref{living}$ below, which gives a relation between the values of the Rasmussen invariants of $L(D, 2t)$ and $L(D, 2t+1)$. 
\begin{defn}
We call an invariant $\nu$ of a {\it Livingston-Naik type} if $\nu$ is an integer-valued additive knot invariant which bounds the smooth 4-genus of a knot and coincides with the $4$-ball genera of positive torus knots, that is, 
\begin{itemize}
\item $\nu$ is a homomorphism from the smooth knot concordance group $\mathcal{C}$ to $\mathbf{Z}$, 
\item $|\nu(K)|\leq g_{4}(K)$, where $g_{4}(K)$ is the $4$-genus of a knot $K$, 
\item $\nu(T_{p, q})=(p-1)(q-1)/2$, where $p$ and $q$ are coprime integers. 
\end{itemize}
\end{defn}
\begin{rem}
For example the Ozsv{\' a}th-Szab{\' o} invariant $\tau$ and a half of the Rasmussen invariant $s/2$ are Livingston-Naik type invariants. 
\end{rem}
\begin{thm}[{\cite[Theorem~$1$]{concordance_doubled}}]\label{living}
Let $\nu$ be a Livingston-Naik type invariant. 
If $\nu(L(D, 2t))=\pm 1$, then $\nu(L(D, 2t+1))=0$. 
\end{thm}
\begin{rem}
In their paper, Livingston and Naik use the notation $D_{-}(K, t)$ and $D_{+}(K, t)$ instead of $L(D, 2t-2f)$ and $L(D, 2t+1-2f)$ respectively. 
\end{rem}
\par
Theorem~$\ref{living}$ does not determine the value of the Rasmussen invariant of a twisted Whitehead double of a knot. 
From Theorem~$\ref{living}$ and Corollary~$\ref{cor5-1}$, we can compute the Rasmussen invariants of twisted Whitehead doubles of any knot with sufficiently many twists.
\begin{cor}\label{ras_double}
For any $n>l$, we have 
\begin{align*}
s(L(D, 2n))&=-2, \\
s(L(D,2n+1))&=0,\\
s(L(D,-2n))&=0, \\
s(L(D,-2n+1))&=2.
\end{align*}
\end{cor}
\begin{proof}
Let $-D$ be the mirror image of the diagram $D$. 
From Proposition~$\ref{lem4}$, we have $s(L(D, 2n))=-2$. Since $L(D, -2n+1)$ and the mirror image of $L(-D, 2n)$ are diagrams of the same knot, we obtain $s(L(D, -2n+1))=-s(L(-D, 2n))$. 
Since we can apply Proposition~$\ref{lem4}$ to $L(-D, 2n)$, we have $s(L(D, -2n+1))=-s(L(-D, 2n))=2$. 
It follows from Theorem~$\ref{living}$ that $s(L(D, 2n+1))=0=s(L(-D, 2n+1))$. Since $L(D, -2n)$ and the mirror image of $L(-D, 2n+1)$ are diagrams of the same knot, we have $s(L(D,-2n))=0$.
\end{proof}
We can rewrite Corollary~$\ref{ras_double}$ as follows. 
\begin{cor}\label{ras_double2}
For any knot $K$, we have $s(D_{+}(K, t))=0$ for $t>2l_{+}(K)$ and $s(D_{+}(K, t))=2$ for $t<-2l_{-}(K)$, where $l_{+}(K)=\min \{l_{+}(D)|\text{D is a diagram of }K\}$ and $l_{-}(K)=\min \{l_{-}(D)|\text{D is a diagram of }K\}$ {\rm (}see Figure~$\ref{graph1}${\rm )}. 
\end{cor}
\begin{figure}[!h]
\begin{center}
\includegraphics[scale=0.4]{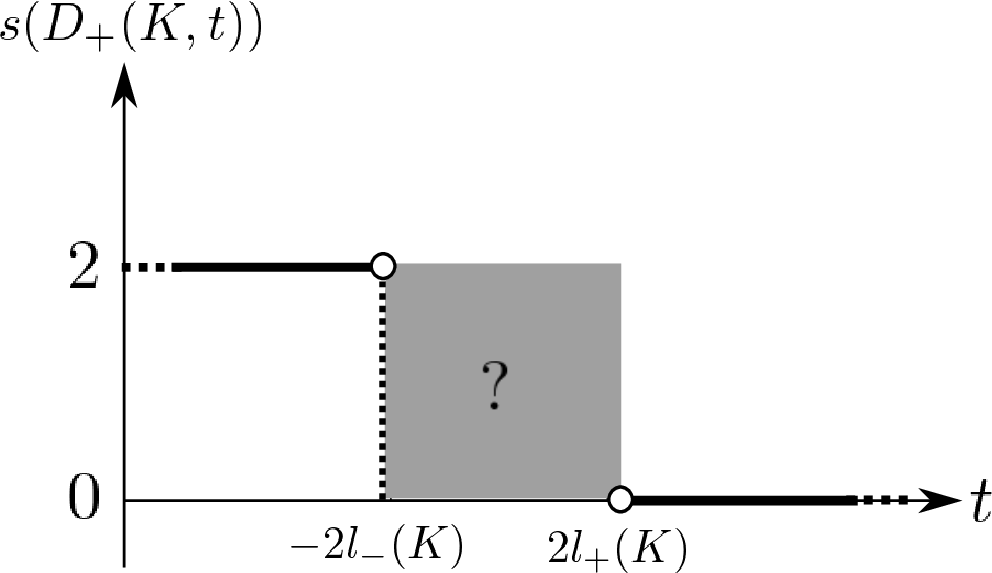}
\end{center}
\caption{$s(D_{+}(K, t))$. }
\label{graph1}
\end{figure}
\begin{rem}
Note that we use a relation between the Khovanov homology and the Rasmussen invariant $s$ in Corollary~$\ref{ras_double2}$ (or Corollary~$\ref{ras_double}$) . 
We do not know whether another Livingston-Naik type invariant satisfies Corollary~$\ref{ras_double2}$ or not. 
\end{rem}
We only compute the Khovanov homology groups of a twisted Whitehead double of any knot with sufficiently many twists. 
Since the Rasmussen invariant $s$ is obtained from the Lee homology, the estimation in Corollary~$\ref{ras_double2}$ may not be sharp. 
Livingston and Naik \cite{concordance_doubled} showed the following theorem which is similar to Corollary~$\ref{ras_double2}$. 
%
\begin{thm}[{\cite[Theorem~$2$]{concordance_doubled}}]\label{living2}
Let $\nu$ be a Livingston-Naik type invariant. 
For each knot $K$, we have $\nu(D_{+}(K, t))=1$ for $t\leq \operatorname{TB}(K)$ and $\nu(D_{+}(K, t))=0$ for $t\geq  -\operatorname{TB}(-K)$, where $\operatorname{TB}(K)$ is the maximal Thurston-Bennequin number of a knot $K$ and $-K$ is the mirror image of $K$. 
\end{thm}
\begin{figure}[!h]
\begin{center}
\includegraphics[scale=0.4]{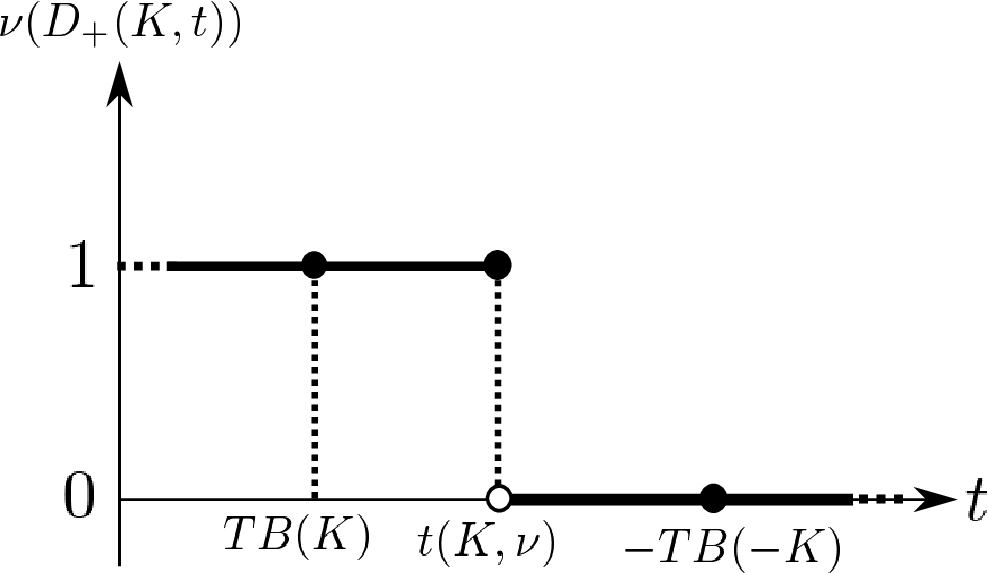}
\end{center}
\caption{$\nu(D_{+}(K, t))$. }
\label{graph2}
\end{figure}
\begin{figure}[!h]
\begin{center}
\includegraphics[scale=0.4]{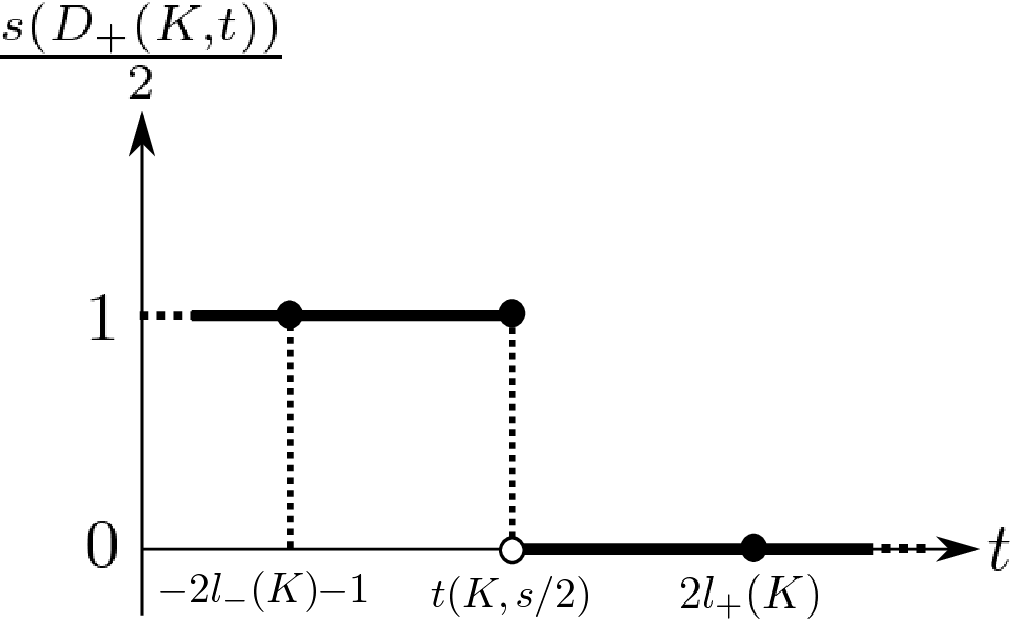}
\end{center}
\caption{$s(D_{+}(K, t))/2$. }
\label{graph3}
\end{figure}
\begin{rem}\label{remark1}
For any Livingston-Naik type invariant $\nu$ and knot $K$, 
Livingston and Naik show that $\nu(D_{+}(K,t))$ is a non-increasing function of $t$. 
Hence, there exists an integer $t(K, \nu)$ such that $\nu(D_{+}(K,t))=1$ for $t\leq t(K, \nu)$ and $\nu(D_{+}(K,t))=0$ for $t>t(K, \nu)$ (see \cite[Theorem~$2$]{concordance_doubled}). \par
From Theorem~$\ref{living2}$, for any Livingston-Naik type invariant $\nu$, we have $TB(K)\leq t(K,\nu)<-TB(-K)$ (Figure~$\ref{graph2}$). 
In particular, we obtain 
\begin{center}
$TB(K)\leq t(K,s/2)<-TB(-K)$. 
\end{center}
From Corollary~$\ref{ras_double2}$, we have 
\begin{center}
$-2l_{-}(K)-1\leq t(K, s/2)\leq 2l_{+}(K)$.  
\end{center}
See also Figure~$\ref{graph3}$. 
As far as the author knows, there is no relation between the maximal Thurston-Bennequin number and the positive or negative crossing number. 
However they have a similar property as above. 
\par
For the Ozsv{\' a}th-Szab{\' o} invariant $\tau$, it is known that $t(K,\tau)=2\tau(K)-1$ (see Theorem~$\ref{hedden}$ below). 
\end{rem}
\begin{example}
For the right-handed trefoil $T_{2,3}$, we have 
$l_{-}(T_{2,3})=0$, $l_{+}(T_{2,3})=3$, 
$\operatorname{TB}(T_{2,3})=1$ and $\operatorname{TB}(-T_{2,3})=-6$. 
From Theorem~$\ref{living2}$, we have $s(D_{+}(T_{2,3}, t))=2$ for $t\leq 1$ and $s(D_{+}(T_{2,3}, t))=0$ for $t\geq 6$. 
From Corollary~$\ref{ras_double2}$, we have $s(D_{+}(T_{2,3}, t))=2$ for $t\leq 1$ and $s(D_{+}(T_{2,3}, t))=0$ for $t\geq 7$. 
Hence, in this case, Theorem~$\ref{living2}$ implies Corollary~$\ref{ras_double2}$. 
However, in general, we do not know whether Theorem~$\ref{living2}$ implies Corollary~$\ref{ras_double2}$ or not. 
\end{example}
\begin{thm}[{\cite[Theorem~$1.4$]{hedden1}}]\label{hedden}
For any knot $K$, we have 
\begin{align*}
\tau(D_{+}(K, t))=
\begin{cases}
0& \text{if } t>2\tau(K)-1, \\
1& \text{if } t\leq 2\tau(K)-1. 
\end{cases}
\end{align*}
\end{thm}
%
%
%
%
\begin{rem}
The negative half of the knot signature $-\sigma/2$ is not of a Livingston-Naik type since $-\sigma(T_{p,q})/2$ is not equal to $(p-1)(q-1)/2$. 
However it has similar properties. 
We call such an invariant of a weak Livingston-Naik type (see Definition~$\ref{weak}$ below). 
\end{rem}
\begin{defn}\label{weak}
We call an invariant $\nu'$ of a {\it weak Livingston-Naik type} if $\nu'$ is an integer-valued additive knot invariant which bounds the smooth 4-genus of a knot and coincides with the $4$-ball genus of right-handed trefoil knot, that is, 
\begin{itemize}
\item $\nu'$ is a homomorphism from the smooth knot concordance group $\mathcal{C}$ to $\mathbf{Z}$, 
\item $|\nu'(K)|\leq g_{4}(K)$, where $g_{4}(K)$ is the $4$-genus of a knot $K$, 
\item $\nu'(T_{2, 3})=1$. 
\end{itemize}
\end{defn}
\begin{rem}
In \cite{abe_tetsuya1}, Abe calls the properties in Definition~$\ref{weak}$ the L-property. 
\end{rem}
\begin{rem}
For any Livingston-Naik type invariant $\nu$, we only use the properties in Definition~$\ref{weak}$ to prove that $\nu(D_{+}(K,t))$ is a non-increasing function of $t$. 
Hence, for any weak Livingston-Naik type invariant $\nu'$ and knot $K$, $\nu'(D_{+}(K,t))$ is a non-increasing function of $t$ and 
there exists an integer $t(K, \nu')$ such that $\nu'(D_{+}(K,t))=1$ for $t\leq t(K, \nu')$ and $\nu'(D_{+}(K,t))=0$ for $t>t(K, \nu')$ (see \cite[Theorem~$2$]{concordance_doubled} and \cite[Corollary~$3$]{living}). 
In particular, the negative half of the knot signature $\sigma$ is of a weak Livingston-Naik type and $t(K, -\sigma/2)=0$. 
\end{rem}
%
\section{Appendix}\label{appendix}
In this section, we prove Claims~$\ref{key_claim}$, $\ref{key_claim2}$ and $\ref{key_claim3}$ and Lemma~$\ref{lem1}$. 
\begin{proof}[Proof of Claim~$\ref{key_claim}$]
To prove Claim~$\ref{key_claim}$, we consider the diagram $E^{m}(2k, 2k(n+f)+j)$. 
If we slide an arc (which is like a ``cap" illustrated in the following figures) of $E^{m}(2k, 2k(n+f)+j)$, the diagram $E^{m}(2k, 2k(n+f)+j)$ may change to one of the four diagrams depicted in Figures~$\ref{case1}$, $\ref{case2}$, $\ref{case3}$ and $\ref{case4}$.
If $E^{m}(2k, 2k(n+f)+j)$ changes to the diagram depicted in Figure~$\ref{case3}$, then 
we continue the isotopic moves as depicted in Figure~$\ref{case3-2}$. 
Similarly, if $E^{m}(2k, 2k(n+f)+j)$ changes to the diagram depicted in Figure~$\ref{case4}$, then we continue the isotopic moves as depicted in Figure~$\ref{case4-2}$. 
No matter in which of the four cases, there are an $h\in\{1, \dots, 2k-2x\}$, an $x\in\{1, \dots, k\}$, an $s\in\{1, \dots, 2k-2x-1\}$ and an $\varepsilon\in\{0, 1\}$ such that $E^{m}(2k, 2k(n+f)+j)$ is equivalent to $D^{s}(2k-2x, (2k-2x)(n+f)+h)\sqcup U_{\varepsilon}$, where $U_{0}$ is a circle in the plane and $U_{1}$ is the empty set. 
We give $E^{m}(2k, 2k(n+f)+j)$ an orientation such that all crossings of $D^{s}(2k-2x, (2k-2x)(n+f)+h)\sqcup U_{\varepsilon}$ are positive. 
We call the diagram $E^{m}(2k, 2k(n+f)+j)$ is of type-$1$, type-$2$, type-$3$ and type-$4$ if it changes to the positive diagram as in Figures~$\ref{case1}$, $\ref{case2}$, $\ref{case3-2}$ and $\ref{case4-2}$, respectively.  
\begin{figure}[!h]
\begin{center}
\includegraphics[scale=0.53]{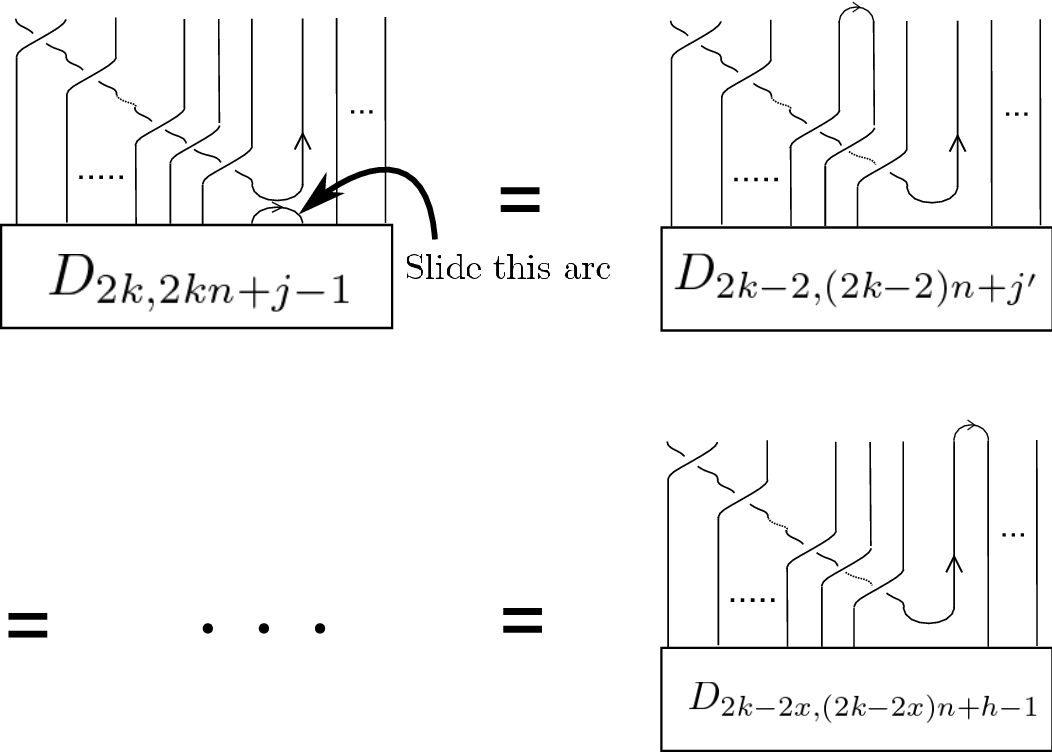}
\end{center}
\caption{The diagram $E^{m}(2k, 2k(n+f)+j)$ can be changed to a positive diagram $D^{s}(2k-2x, (2k-2x)(n+f)+h)\sqcup U_{\varepsilon}$ (type-$1$). }
\label{case1}
\end{figure}
\begin{figure}[!h]
\begin{center}
\includegraphics[scale=0.53]{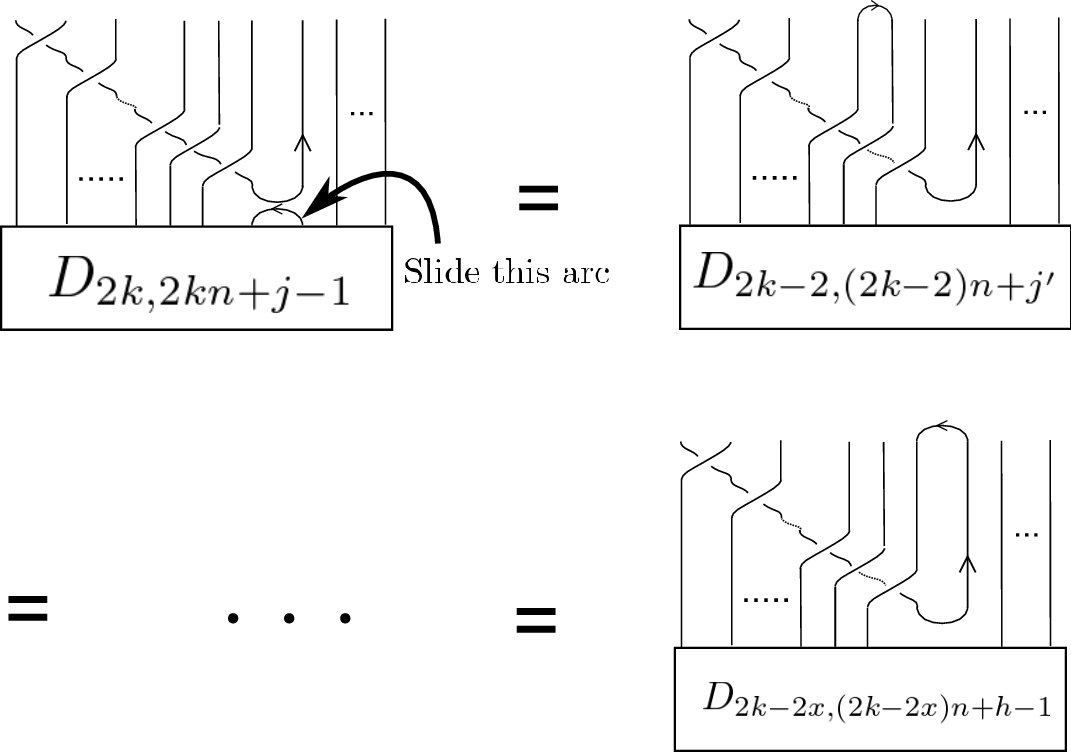}
\end{center}
\caption{The diagram $E^{m}(2k, 2k(n+f)+j)$ can be changed to a positive diagram $D^{s}(2k-2x, (2k-2x)(n+f)+h)\sqcup U_{\varepsilon}$ (type-$2$). }
\label{case2}
\end{figure}
\begin{figure}[!h]
\begin{center}
\includegraphics[scale=0.68]{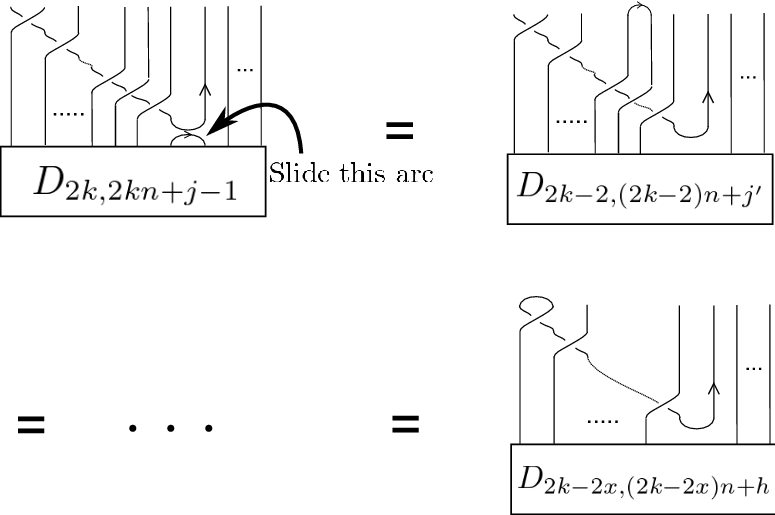}
\end{center}
\caption{The diagram $E^{m}(2k, 2k(n+f)+j)$ can be changed to a diagram (type-$3$). }
\label{case3}
\end{figure}
\begin{figure}[!h]
\begin{center}
\includegraphics[scale=0.53]{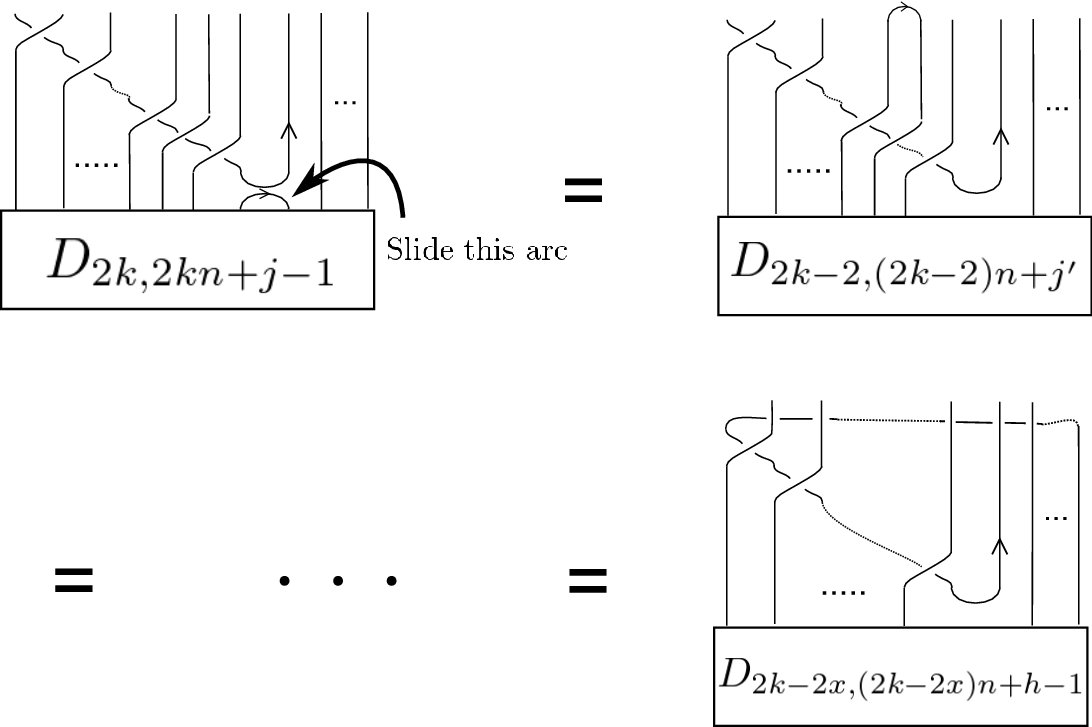}
\end{center}
\caption{The diagram $E^{m}(2k, 2k(n+f)+j)$ can be changed to a diagram (type-$4$). }
\label{case4}
\end{figure}
\begin{figure}[!h]
\begin{center}
\includegraphics[scale=0.6]{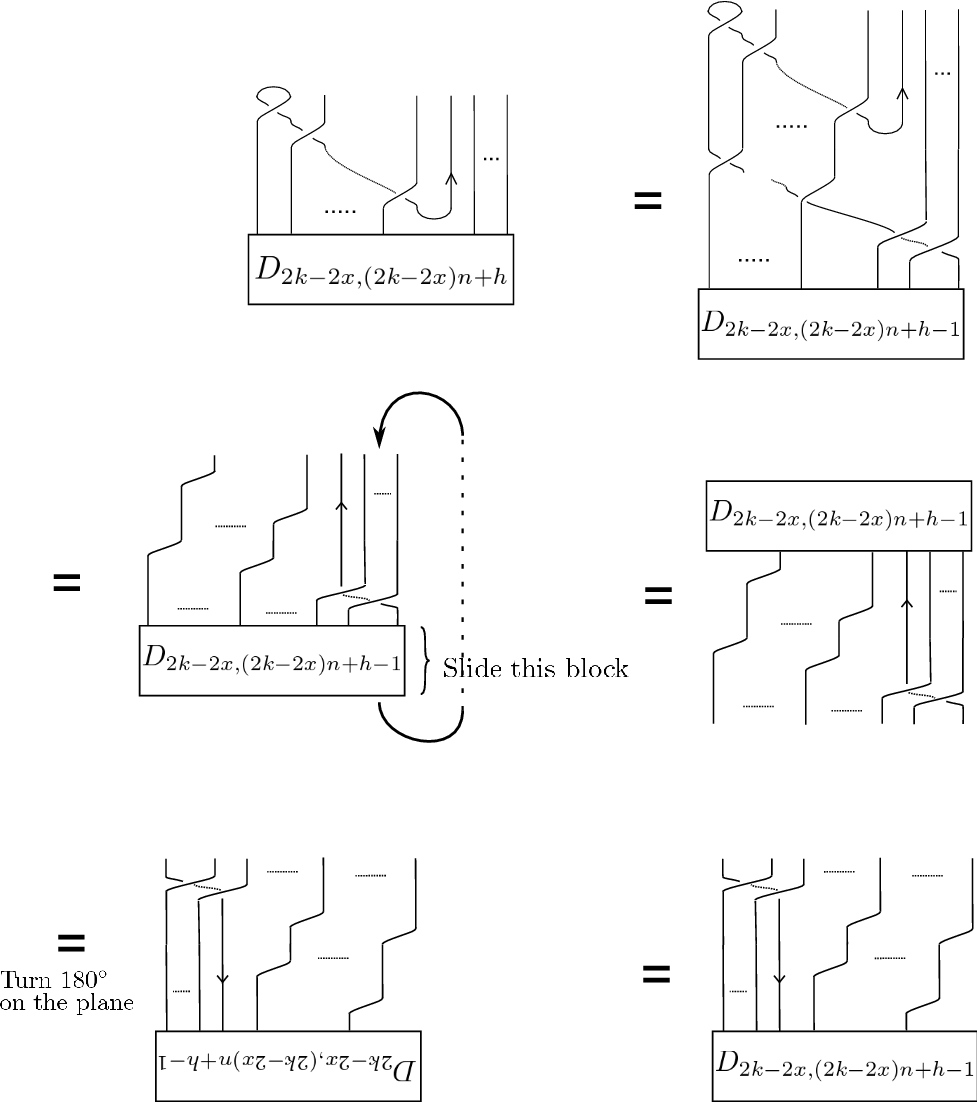}
\end{center}
\caption{The diagram $E^{m}(2k, 2k(n+f)+j)$ can be changed to a positive diagram $D^{s}(2k-2x, (2k-2x)(n+f)+h)\sqcup U_{\varepsilon}$ (type-$3$). }
\label{case3-2}
\end{figure}
\begin{figure}[!h]
\begin{center}
\includegraphics[scale=0.5]{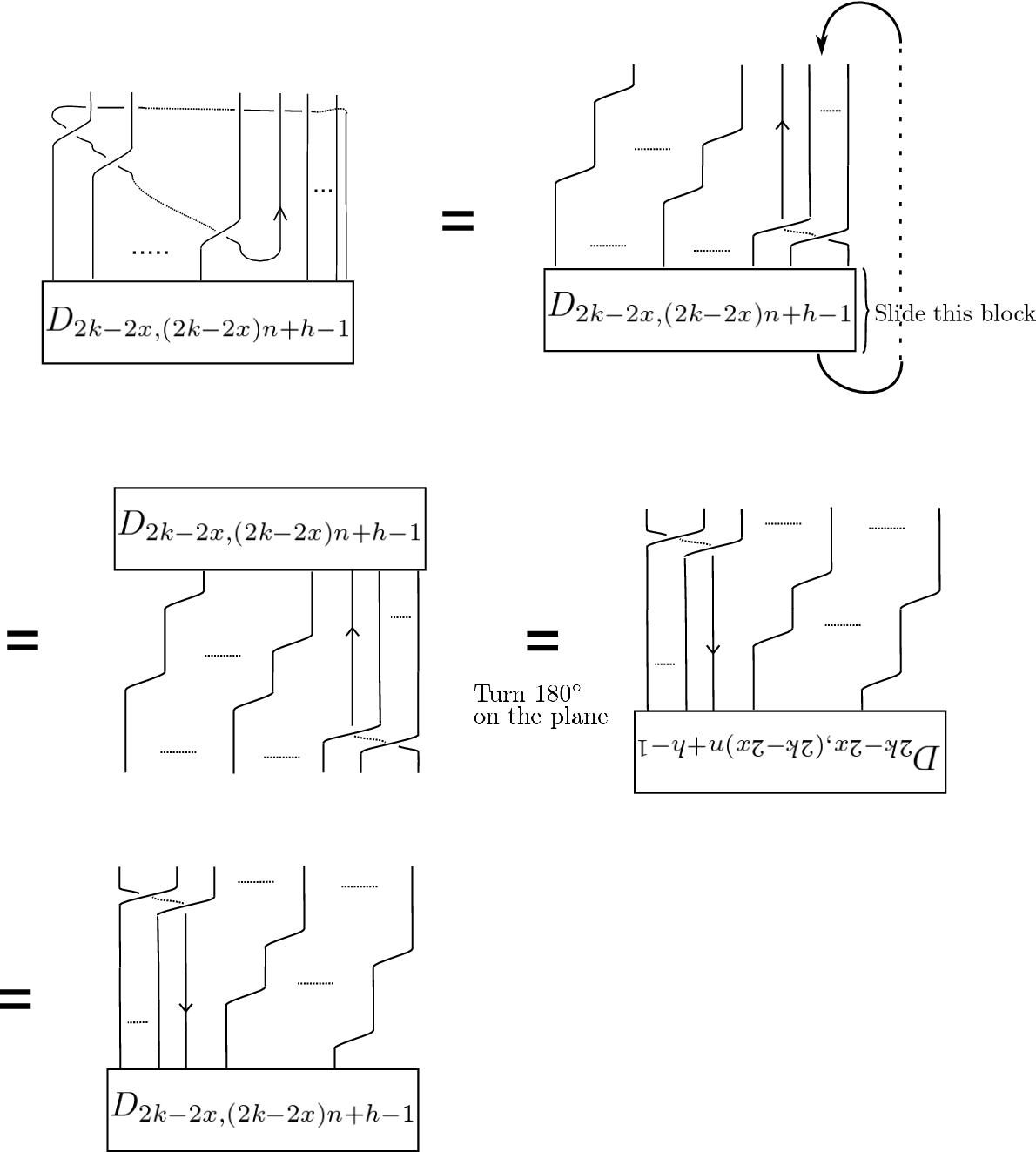}
\end{center}
\caption{The diagram $E^{m}(2k, 2k(n+f)+j)$ can be changed to $D^{s}(2k-2x, (2k-2x)(n+f)+h)\sqcup U_{\varepsilon}$ (type-$4$). }
\label{case4-2}
\end{figure}
\par
Now we have supposed that for $1\leq g<k$, $j=1, \dots, 2g$ and $m=0, \dots, 2g-1$ we have $H^{i}(D^{m}(2g, 2g(n+f)+j))=0$ if 
$i>2g^{2}(n-l+1)+l(2g)^{2}$ and $n\geq l$, or $i>l(2g)^{2}$ and $n<l$ (recall the induction hypothesis in the proof of Lemma~$\ref{lem_cable}$ $(1)$).  
From this induction hypothesis, if $i-n_{-}+l_{-}(2k-2x)^{2}>2(k-x)^{2}(n-l+1)+l(2k-2x)^{2}$ and $n\geq l$, or $i-n_{-}+l_{-}(2k-2x)^{2}>l(2k-2x)^{2}$ and $n<l$, then we have 
\begin{align*}
&H^{i}(E^{m}(2k, 2k(n+f)+j))\\
&=\KH^{i-n_{-}}(D^{s}(2k-2x, (2k-2x)(n+f)+h)\sqcup U_{\varepsilon})\\
&=H^{i-n_{-}+l_{-}(2k-2x)^{2}}(D^{s}(2k-2x, (2k-2x)(n+f)+h)\sqcup U_{\varepsilon})
=0, 
\end{align*}
where $n_{-}$ is the number of the negative crossings of $E^{m}(2k, 2k(n+f)+j)$. 
Hence, to prove Claim~$\ref{key_claim}$, it is sufficient to prove the following: 
\begin{align}
l(2k)^{2}+2k^{2}(n-l+1)-1&\geq 2(k-x)^{2}(n-l+1)\label{e1}\\
&\ \ \ +l_{+}(2k-2x)^{2}+n_{-}  \ \ (n\geq l), \nonumber\\
l(2k)^{2}-1&\geq l_{+}(2k-2x)^{2}+n_{-}  \ \ (n<l).\label{e2} 
\end{align}
To prove $(\ref{e1})$ and $(\ref{e2})$, 
we need to count the number of the negative crossings of $E^{m}(2k, 2k(n+f)+j)$. 
We first count its positive crossings by dividing it into four parts, part-$1$, part-$2$, part-$3$ and part-$4$ (see Figure~$\ref{4-parts}$). 
\begin{figure}[!h]
\begin{center}
\includegraphics[scale=0.45]{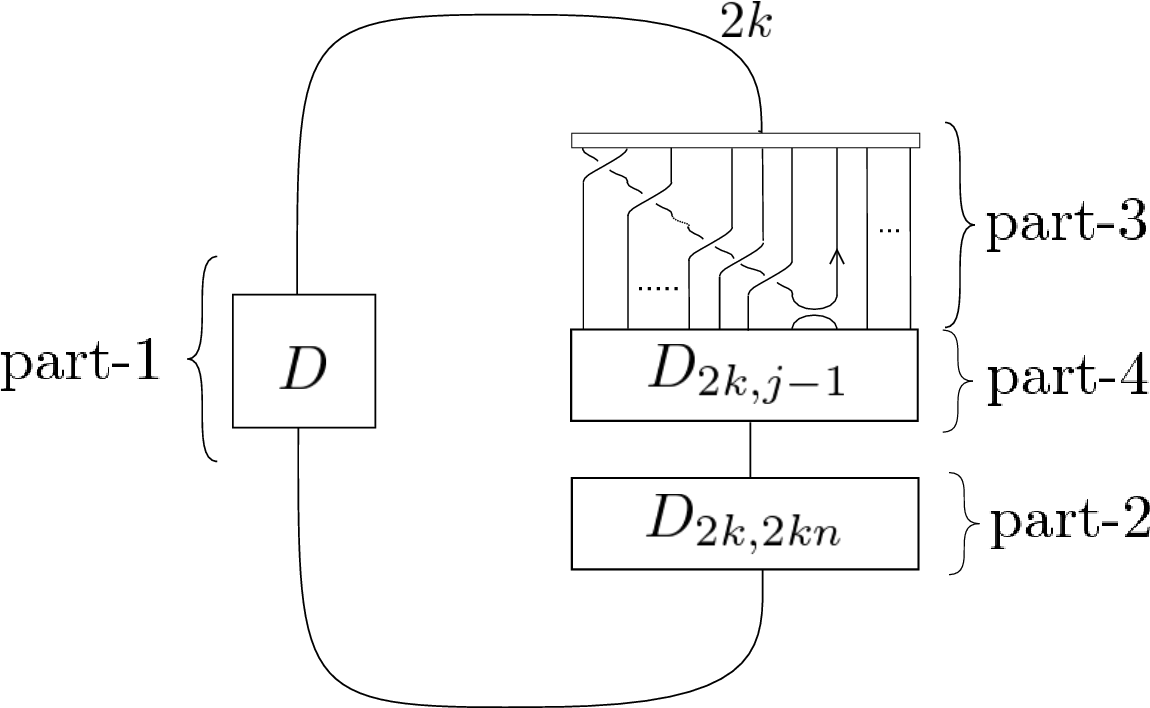}
\end{center}
\caption{The diagram $E^{m}(2k, 2k(n+f)+j)$ divided into four parts. }
\label{4-parts}
\end{figure}
\par
[Step~$1$] In the case where $E^{m}(2k, 2k(n+f)+j)$ is either type-$1$ or type-$2$:  
In part-$1$, we apply $\sum_{i=0}^{x-1}(l(2k-2i)+l(2k-2i-2))$ RII moves to $E^{m}(2k, 2k(n+f)+j)$ to obtain the diagram $D^{s}(2k-2x, (2k-2x)(n+f)+h)\sqcup U_{\varepsilon}$. 
Then $E^{m}(2k, 2k(n+f)+j)$ loses $\sum_{i=0}^{x-1}(l(2k-2i)+l(2k-2i-2))$ positive crossings. 
Moreover, $D^{s}(2k-2x, (2k-2x)(n+f)+h)\sqcup U_{\varepsilon}$ has $l_{+}(2k-2x)^{2}$ positive crossings in a part corresponding to part-$1$. 
Hence, in part-$1$, $E^{m}(2k, 2k(n+f)+j)$ has 
\begin{center}
$\displaystyle{\sum_{i=0}^{x-1}(l(2k-2i)+l(2k-2i-2))+l_{+}(2k-2x)^{2}}$
\end{center}
positive crossings. 
%
\par
In part-$2$, $E^{m}(2k, 2k(n+f)+j)$ has $x$ arcs directed upward and $2k-x$ arcs directed downward (see Figure~$\ref{part2-1}$). 
Hence, in part-$2$, $E^{m}(2k, 2k(n+f)+j)$ has 
$x(x-1)n+(2k-x)(2k-x-1)n$
positive crossings. 
\begin{figure}[!h]
\begin{center}
\includegraphics[scale=0.45]{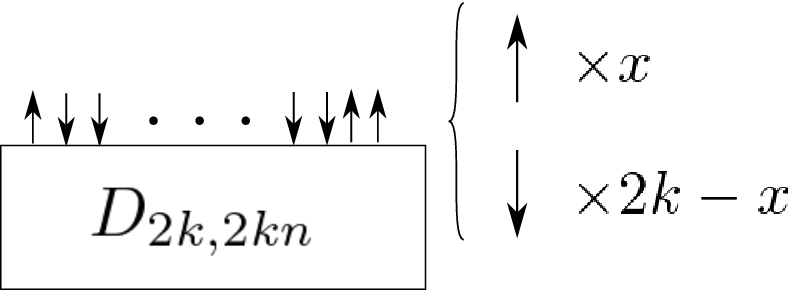}
\end{center}
\caption{If $E^{m}(2k, 2k(n+f)+j)$ is either type-$1$ or type-$2$, in part-$2$, $E^{m}(2k, 2k(n+f)+j)$ has $x$ arcs directed upward and $2k-x$ arcs directed downward. }
\label{part2-1}
\end{figure}
\par
\begin{figure}[!h]
\begin{center}
\includegraphics[scale=0.6]{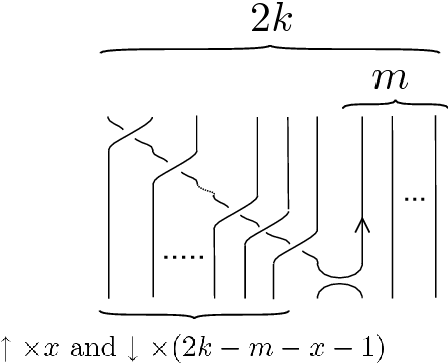}
\end{center}
\caption{If $E^{m}(2k, 2k(n+f)+j)$ is either type-$1$ or type-$2$, in part-$3$, $E^{m}(2k, 2k(n+f)+j)$ has at least $2k-m-1-x$ positive crossings. This figure is a minimal case. }
\label{part3-1}
\end{figure}
In part $3$, $E^{m}(2k, 2k(n+f)+j)$ has at least 
$2k-m-1-x$ 
positive crossings (see Figure~$\ref{part3-1}$). 
\par
In part-$4$, note that there are $x$ arcs directed upward and $2k-x$ arcs directed downward. 
Assume that $b$ is the number of the positions where the left most arc is directed upward and that $a$ is the number of the positions where the left most arc is directed downward  (see Figure~$\ref{part4}$). Note that $a+b=j-1$ and that $b\leq x$. 
\begin{figure}[!h]
\begin{center}
\includegraphics[scale=0.6]{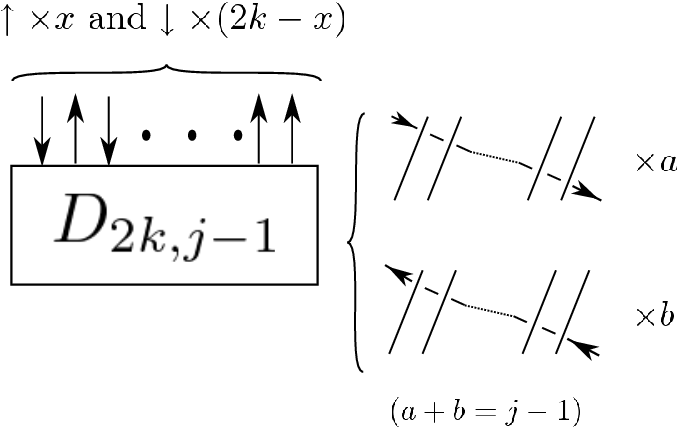}
\end{center}
\caption{In the case where the diagram $E^{m}(2k, 2k(n+f)+j)$ is type-$1$ or type-$2$. In part-$4$, $E^{m}(2k, 2k(n+f)+j)$ has $x$ arcs directed upward and $2k-x$ arcs directed downward. 
The number of the positions where the left most arc is directed upward is $b$. The number of the positions where the left most arc is directed downward is $a$. }
\label{part4}
\end{figure}
Then, in part-4, $E^{m}(2k, 2k(n+f)+j)$ has
\begin{center}
$b(x-1)+a(2k-x-1)=b(x-1)+(j-1-b)(2k-x-1)$ 
\end{center}
positive crossings.  
\par
Hence the diagram $E^{m}(2k, 2k(n+f)+j)$ has at least $X_{1}$ positive crossings, where 
\begin{align*}
X_{1}&=\sum_{i=0}^{x-1}(l(2k-2i)+l(2k-2i-2))+l_{+}(2k-2x)^2\\
&\ \ \ +x(x-1)n+(2k-x)(2k-x-1)n\\
&\ \ \ +2k-1-m-x\\
&\ \ \ +b(x-1)+(j-1-b)(2k-x-1). 
\end{align*}
From the above discussion $E^{m}(2k, 2k(n+f)+j)$ has at most $X_{2}$ negative crossings, 
where
\begin{align*}
X_{2}=l(2k)^{2}+(2k-1)(2kn+j)-m-X_{1}. 
\end{align*}
\par
Then for $j\neq 2k$ we can check the following. 
\begin{align*}
l(2k)^{2}+2k^{2}(n-l+1)-1&\geq 2(k-x)^{2}(n-l+1)+l_{+}(2k-2x)^{2}+X_{2}  \ (n\geq l), \\
l(2k)^{2}-1&\geq l_{+}(2k-2x)^{2}+X_{2}  \ (n<l). 
\end{align*}
\par
Indeed, we can compute $l(2k)^{2}+2k^{2}(n-l+1)-1-(2(k-x)^{2}(n-l+1)+l_{+}(2k-2x)^{2}+X_{2})=2(k-x)(x-b)+x(2k-j)-1$. 
We obtain $2(k-x)(x-b)+x(2k-j)-1\geq 0$ since $0<j<2k$, $b\leq x\leq k$ and $x\geq 1$. 
Similarly $l_{+}(2k-2x)^{2}+X_{2}\leq l(2k)^{2}-1$ for $j\neq 2k$. 
This implies that $(\ref{e1})$ and $(\ref{e2})$ are true if $j\neq 2k$ and $E^{m}(2k, 2k(n+f)+j)$ is either type-$1$ or type-$2$. 
\par
Finally we consider the case where $j=2k$. If $j= 2k$, then $x=1$ and $E^{m}(2k, 2k(n+f)+j)$ has $n_{-}=2(2k-1)(n+1)-1+2l_{+}(2k-1)+l_{-}((2k)^{2}-2(2k-1))$ negative crossings. 
%
%
In this case we have $l_{+}(2k-2)^{2}+2(k-1)^{2}(n-l+1)+n_{-}= l(2k)^{2}+2k^{2}(n-l+1)-1$. 
Similarly, in this case, we obtain $l(2k)^{2}-1\geq l_{+}(2k-2x)^{2}+n_{-}$ for $n<l$. 
These imply that $(\ref{e1})$ and $(\ref{e2})$ are true for $j=2k$.
\par
[Step~$2$] In the case where $E^{m}(2k, 2k(n+f)+j)$ is either type-$3$ or type-$4$: 
By the same discussion, in part-$1$, $E^{m}(2k, 2k(n+f)+j)$ has 
\begin{center}
$\displaystyle{\sum_{i=0}^{x-1}(l(2k-2i)+l(2k-2i-2))+l_{+}(2k-2x)^{2}}$
\end{center}
positive crossings. 
\par
In part-$2$, $E^{m}(2k, 2k(n+f)+j)$ has $2k-x$ arcs directed upward and $x$ arcs directed downward (see Figure~$\ref{part2-2}$). 
Hence, in part-$2$, $E^{m}(2k, 2k(n+f)+j)$ has 
$x(x-1)n+(2k-x)(2k-x-1)n$
positive crossings. 
\begin{figure}[!h]
\begin{center}
\includegraphics[scale=0.45]{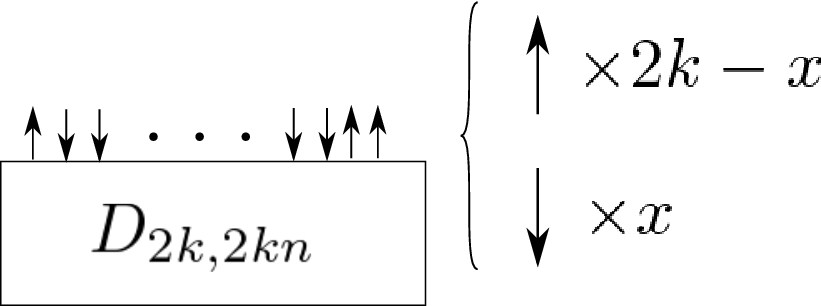}
\end{center}
\caption{If $E^{m}(2k, 2k(n+f)+j)$ is either type-$3$ or type-$4$, in part-$2$, $E^{m}(2k, 2k(n+f)+j)$ has $2k-x$ arcs directed upward and $x$ arcs directed downward. }
\label{part2-2}
\end{figure}
\par
In part-$3$, $E^{m}(2k, 2k(n+f)+j)$ may have no positive crossing. 
\par
In part-$4$, note that there are $2k-x$ arcs directed upward and $x$ arcs directed downward. 
Assume that $a$ is the number of the positions where the left most arc is directed upward and that $b$ is the number of the positions where the left most arc is directed downward  (see Figure~$\ref{part4-2}$). Note that $a+b=j-1$ and that $b< x$ (we have $b\neq x$ since in part-$4$ the left most bottom arc is directed downward ). 
\begin{figure}[!h]
\begin{center}
\includegraphics[scale=0.6]{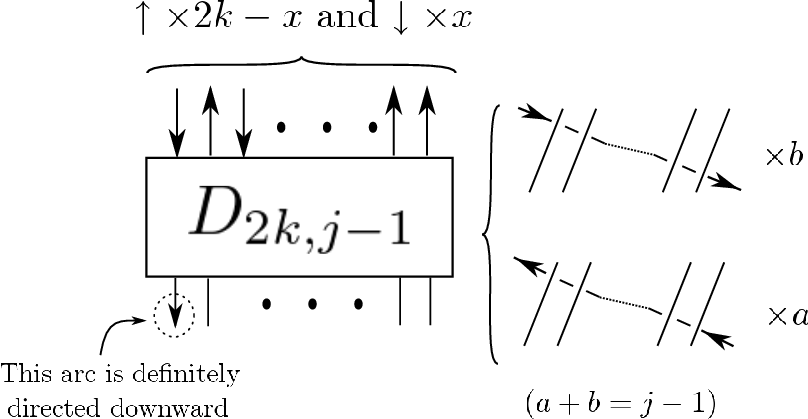}
\end{center}
\caption{In the case where the diagram $E^{m}(2k, 2k(n+f)+j)$ is type-$3$ or type-$4$. In part-$4$, $E^{m}(2k, 2k(n+f)+j)$ has $2k-x$ arcs directed upward and $x$ arcs directed downward. 
The number of the positions where the left most arc is directed upward is $a$. The number of the positions where the left most arc is directed downward is $b$. 
The left most bottom arc is directed downward since we give $E^{m}(2k, 2k(n+f)+j)$ such an orientation, (see Figures~$\ref{case3-2}$, $\ref{case4-2}$ or $\ref{4-parts}$). }
\label{part4-2}
\end{figure}
Then, in part-4, $E^{m}(2k, 2k(n+f)+j)$ has
\begin{center}
$b(x-1)+a(2k-x-1)=b(x-1)+(j-1-b)(2k-x-1)$ 
\end{center}
positive crossings.  
\par
Hence the diagram $E^{m}(2k, 2k(n+f)+j)$ has at least $X'_{1}$ positive crossings, where 
\begin{align*}
X'_{1}&=\sum_{i=0}^{x-1}(l(2k-2i)+l(2k-2i-2))+l_{+}(2k-2x)^2\\
&\ \ \ +x(x-1)n+(2k-x)(2k-x-1)n\\
&\ \ \ +b(x-1)+(j-1-b)(2k-x-1). 
\end{align*}
From the above discussion, $E^{m}(2k, 2k(n+f)+j)$ has at most $X'_{2}$ negative crossings, 
where
\begin{align*}
X'_{2}=l(2k)^{2}+(2k-1)(2kn+j)-m-X'_{1}. 
\end{align*}
Then for $j\neq 2k$ we can also check the following. 
\begin{align*}
l(2k)^{2}+2k^{2}(n-l+1)-1&\geq 2(k-x)^{2}(n-l+1)+l_{+}(2k-2x)^{2}+X'_{2}  \ (n\geq l), \\
l(2k)^{2}-1&\geq l_{+}(2k-2x)^{2}+X'_{2}  \ (n<l). 
\end{align*}
Indeed, we can compute $l(2k)^{2}+2k^{2}(n-l+1)-1-(2(k-x)^{2}(n-l+1)+l_{+}(2k-2x)^{2}+X'_{2})=2(k-x)(x-b-1)+x(2k-j-1)+m$. 
We obtain $2(k-x)(x-b-1)+x(2k-j-1)+m\geq m>0$ since we have $0<j<2k$, $b< x\leq k$ and $x\geq 1$. 
Similarly $l_{+}(2k-2x)^{2}+X'_{2}\leq l(2k)^{2}-1$. 
\par
From Steps~$1$ and $2$, we finish this proof. 

\end{proof}
\begin{proof}[Proof of Claim~$\ref{key_claim2}$]
The proof of Claim~$\ref{key_claim2}$ is the same as that of Claim~$\ref{key_claim}$. 
\par
By the same discussion, there are an $h\in\{1, \dots, 2k+1-2x\}$, an $x\in\{1, \dots, k\}$, an $s\in\{1, \dots, 2k-2x\}$ and an $\varepsilon\in\{0, 1\}$ such that $E^{m}(2k+1, (2k+1)(n+f)+j)$ is equivalent to $D^{s}(2k+1-2x, (2k+1-2x)(n+f)+h)\sqcup U_{\varepsilon}$, where $U_{0}$ is a circle in the plane and $U_{1}$ is the empty set. 
We give $E^{m}(2k+1, (2k+1)(n+f)+j)$ an orientation such that all crossings of $D^{s}(2k+1-2x, (2k+1-2x)(n+f)+h)\sqcup U_{\varepsilon}$ are positive. 
\par
Now we have supposed that for $1\leq g<k$, $j=1, \dots, 2g+1$ and $m=0, \dots, 2g$ we have $H^{i}(D^{m}(2g+1, (2g+1)(n+f)+j))=0$ if 
$i>2g(g+1)(n-l+1)+l(2g+1)^{2}$ and $n\geq l$, or $i>l(2g+1)^{2}$ and $n<l$ (recall the induction hypothesis in the proof of Lemma~$\ref{lem_cable}$ $(2)$).  
From this induction hypothesis, if $i-n_{-}+l_{-}(2k+1-2x)^{2}>2(k-x)(k-x+1)(n-l+1)+l(2k+1-2x)^{2}$ and $n\geq l$, or $i-n_{-}+l_{-}(2k+1-2x)^{2}>l(2k+1-2x)^{2}$ and $n<l$, then we have 
\begin{align*}
&H^{i}(E^{m}(2k+1, (2k+1)(n+f)+j))
=0, 
\end{align*}
where $n_{-}$ is the number of the negative crossings of $E^{m}(2k+1, (2k+1)(n+f)+j)$. 
Hence, to prove Claim~$\ref{key_claim2}$, it is sufficient to prove the following: 
\begin{align}
l(2k+1)^{2}+2k(k+1)(n-l+1)-1&\geq 2(k-x)(k+1-x)(n-l+1)\label{e3}\\
&\ \ \ +l_{+}(2k+1-2x)^{2}+n_{-}  \ \ (n\geq l), \nonumber\\
l(2k+1)^{2}-1&\geq l_{+}(2k+1-2x)^{2}+n_{-}  \ \ (n<l).\label{e4} 
\end{align}
To prove $(\ref{e3})$ and $(\ref{e4})$, 
we need to count the number of the negative crossings of $E^{m}(2k+1, (2k+1)(n+f)+j)$. 
We first count its positive crossings by dividing it into four parts as the proof of Claim~$\ref{key_claim}$. 
%
%
\par
[Step~$1$] In the case where $E^{m}(2k+1, (2k+1)(n+f)+j)$ is either type-$1$ or type-$2$:  
In part-$1$, $E^{m}(2k+1, (2k+1)(n+f)+j)$ has 
\begin{center}
$\displaystyle{\sum_{i=0}^{x-1}(l(2k+1-2i)+l(2k-2i-1))+l_{+}(2k+1-2x)^{2}}$
\end{center}
positive crossings. 
%
\par
In part-$2$, $E^{m}(2k+1, (2k+1)(n+f)+j)$ has 
$x(x-1)n+(2k+1-x)(2k-x)n$
positive crossings. 
\par
In part $3$, $E^{m}(2k+1, (2k+1)(n+f)+j)$ has at least 
$2k-m-x$ 
positive crossings (cf. Figure~$\ref{part3-1}$). 
\par
In part-$4$, $E^{m}(2k+1, (2k+1)(n+f)+j)$ has
%
$b(x-1)+a(2k-x)=b(x-1)+(j-1-b)(2k-x)$ 
positive crossings.  
\par
Hence the diagram $E^{m}(2k+1, (2k+1)(n+f)+j)$ has at least $Y_{1}$ positive crossings, where 
\begin{align*}
Y_{1}&=\sum_{i=0}^{x-1}(l(2k+1-2i)+l(2k-2i-1))+l_{+}(2k+1-2x)^2\\
&\ \ \ +x(x-1)n+(2k+1-x)(2k-x)n\\
&\ \ \ +2k-m-x\\
&\ \ \ +b(x-1)+(j-1-b)(2k-x). 
\end{align*}
From the above discussion, $E^{m}(2k+1, (2k+1)(n+f)+j)$ has at most $Y_{2}$ negative crossings, 
where
\begin{align*}
Y_{2}=l(2k+1)^{2}+2k((2k+1)n+j)-m-Y_{1}. 
\end{align*}
\par
Then for $j\neq 2k+1$ we can check the following. 
\begin{align*}
l(2k+1)^{2}+2k(k+1)(n-l+1)-1&\geq 2(k-x)(k-x+1)(n-l+1)\\
&\ \ \ +l_{+}(2k+1-2x)^{2}+Y_{2}  \ (n\geq l), \\
l(2k+1)^{2}-1&\geq l_{+}(2k+1-2x)^{2}+Y_{2}  \ (n<l). 
\end{align*}
Finally we consider the case where $j=2k+1$. If $j= 2k+1$ then $x=1$ and $E^{m}(2k+1, (2k+1)(n+f)+j)$ has $n_{-}=4k(n+1)-1+4l_{+}k+l_{-}((2k+1)^{2}-4k)$ negative crossings. 
In this case we have $l_{+}(2k-1)^{2}+2k(k-1)(n-l+1)+n_{-}= l(2k+1)^{2}+2k(k+1)(n-l+1)-1$. 
Similarly, in this case, we obtain $l(2k+1)^{2}-1\geq l_{+}(2k+1-2x)^{2}+n_{-}$ for $n<l$. 
These imply that $(\ref{e3})$ and $(\ref{e4})$ are true for $j=2k+1$.
\par
[Step~$2$] In the case where $E^{m}(2k+1, (2k+1)(n+f)+j)$ is either type-$3$ or type-$4$: 
\par
By the same discussion, in part-$1$, $E^{m}(2k+1, (2k+1)(n+f)+j)$ has 
\begin{center}
$\displaystyle{\sum_{i=0}^{x-1}(l(2k+1-2i)+l(2k-2i-1))+l_{+}(2k+1-2x)^{2}}$
\end{center}
positive crossings. 
\par
In part-$2$, $E^{m}(2k+1, (2k+1)(n+f)+j)$ has 
$x(x-1)n+(2k+1-x)(2k-x)n$
positive crossings. 
%
\par
In part-$3$, $E^{m}(2k+1, (2k+1)(n+f)+j)$ may have no positive crossing. 
\par
In part-$4$, $E^{m}(2k+1, (2k+1)(n+f)+j)$ has
$b(x-1)+a(2k-x)=b(x-1)+(j-1-b)(2k-x)$ 
positive crossings.  
\par
Hence the diagram $E^{m}(2k+1, (2k+1)(n+f)+j)$ has at least $Y'_{1}$ positive crossings, where 
\begin{align*}
Y'_{1}&=\sum_{i=0}^{x-1}(l(2k+1-2i)+l(2k-2i-1))+l_{+}(2k+1-2x)^2\\
&\ \ \ +x(x-1)n+(2k+1-x)(2k-x)n\\
&\ \ \ +b(x-1)+(j-1-b)(2k-x). 
\end{align*}
From the above discussion, $E^{m}(2k+1, (2k+1)(n+f)+j)$ has at most $Y'_{2}$ negative crossings, 
where
\begin{align*}
Y'_{2}=l(2k+1)^{2}+2k((2k+1)n+j)-m-Y'_{1}. 
\end{align*}
\par
Then for $j\neq 2k+1$ we can also check the following. 
\begin{align*}
l(2k+1)^{2}+2k(k+1)(n-l+1)-1&\geq 2(k-x)(k-x+1)(n-l+1)\\
&\ \ \ +l_{+}(2k+1-2x)^{2}+Y'_{2}  \ (n\geq l), \\
l(2k+1)^{2}-1&\geq l_{+}(2k+1-2x)^{2}+Y'_{2}  \ (n<l). 
\end{align*}

From Steps~$1$ and $2$, we finish this proof. 
\end{proof}
\begin{proof}[Proof of Claim~$\ref{key_claim3}$]
In the proof of Claim~$\ref{key_claim}$, we have proved that 
\begin{itemize}
\item there are an $h\in\{1, \dots, 2k-2x\}$, an $x\in\{1, \dots, k\}$, an $s\in\{1, \dots, 2k-2x-1\}$ and an $\varepsilon\in\{0, 1\}$ such that $E^{m}(2k, 2k(n+f)+j)$ is equivalent to $D^{s}(2k-2x, (2k-2x)(n+f)+h)\sqcup U_{\varepsilon}$, where $U_{0}$ is a circle in the plane and $U_{1}$ is the empty set, 
\item if $E^{m}(2k, 2k(n+f)+j)$ is either type-$1$ or type-$2$, then $E^{m}(2k, 2k(n+f)+j)$ has at most $X_{2}$ negative crossings,  
\item if $E^{m}(2k, 2k(n+f)+j)$ is either type-$3$ or type-$4$, then $E^{m}(2k, 2k(n+f)+j)$ has at most $X'_{2}$ negative crossings. 
\end{itemize}
From Lemma~$\ref{lem_cable}$, if $i-n_{-}+l_{-}(2k-2x)^{2}>2(k-x)^{2}(n-l+1)+l(2k-2x)^{2}$ and $n\geq l$, then we have 
\begin{align*}
H^{i}(E^{m}(2k, 2k(n+f)+j))
=H^{i-n_{-}+l_{-}(2k-2x)^{2}}(D^{s}(2k-2x, (2k-2x)(n+f)+h)\sqcup U_{\varepsilon})
=0, 
\end{align*}
where $n_{-}$ is the number of the negative crossings of $E^{m}(2k, 2k(n+f)+j)$. 
In particular, if $i>2(k-x)^{2}(n-l+1)+l_{+}(2k-2x)^{2}+n_{-}$ and $n\geq l$, then we have 
\begin{align*}
H^{i}(E^{m}(2k, 2k(n+f)+j))=0.  
\end{align*}
From the above results, to prove Claim~$\ref{key_claim3}$, it is sufficient to prove that 
\begin{enumerate}
\item if $E^{m}(2k, 2k(n+f)+j)$ is either type-$1$ or type-$2$, then 
$l(2k)^{2}+2k^{2}(n-l)-2\geq 2(k-x)^{2}(n-l+1)+l_{+}(2k-2x)^{2}+X_{2}$, \label{claim3-(1)}
\item if $E^{m}(2k, 2k(n+f)+j)$ is either type-$3$ or type-$4$, then 
$l(2k)^{2}+2k^{2}(n-l)-2\geq 2(k-x)^{2}(n-l+1)+l_{+}(2k-2x)^{2}+X'_{2}$. \label{claim3-(2)}
\end{enumerate}
We have already proved $(\ref{claim3-(2)})$ in the proof of Claim~$\ref{key_claim}$. Let us prove $(\ref{claim3-(1)})$. 
Recall $j=1, \dots, 2k-1$, $b\leq x\leq k$ and $x\geq 1$. 
Hence, if $j\leq 2k-2$ or $x\geq 2$, we obtain 
\begin{align*}
&l(2k)^{2}+2k^{2}(n-l)-2-(2(k-x)^{2}(n-l+1)+l_{+}(2k-2x)^{2}+X_{2})\\
&=-2+x(2k-j)+2(k-x)(b-x)\geq 0. \\
\end{align*} 
If $j=2k-1$ and $x=1$, then $E^{m}(2k, 2k(n+f)+j)$ is either type-$3$ or type-$4$. 
Hence we obtain 
$l(2k)^{2}+2k^{2}(n-l)-2-( 2(k-x)^{2}(n-l+1)+l_{+}(2k-2x)^{2}+X_{2})\geq 0$ 
if $E^{m}(2k, 2k(n+f)+j)$ is either type-$1$ or type-$2$. 
\end{proof}
%
%
%
\begin{proof}[Proof of Lemma~$\ref{lem1}$]
To prove Lemma~$\ref{lem1}$, we use Lemma~$\ref{lem1_lem}$ below. 
It follows from Lemma~$\ref{lem1_lem}$ that  
\begin{align*}
H^{2k(k+1)n}(D_{2k+1, (2k+1)n-1})=H^{2k(k+1)n}(D_{2k+1, (2k+1)(n-1)}) 
\end{align*} 
for any positive integers $n$ and $k$. 
From Lemma~$\ref{mainthm3}$, the right hand side is zero. 
\end{proof}
\begin{lem}\label{lem1_lem}
Let $K$ be a knot and $D$ be a knot diagram with $l_{+}$ positive crossings and $l_{-}$ negative crossings. 
Put $l=l_{+}+l_{-}$ and $f=l_{+}-l_{-}$. 
Then for any positive integer $k$ and any $n>l$, we obtain
\begin{align*}
&H^{2k(k+1)(n+l)+l}(D(2k+1, (2k+1)(n+f)-1))\\
&=H^{2k(k+1)(n+l)+l}(D(2k+1, (2k+1)(n+f-1))). 
\end{align*} 
\end{lem}
\begin{proof}
We first compute $H^{i}(E^{m}(2k+1, (2k+1)(n+f-1)+j))$. 
In the proof of Claim~$\ref{key_claim2}$, we have proved that 
\begin{itemize}
\item there are an $h\in\{1, \dots, 2k+1-2x\}$, an $x\in\{1, \dots, k\}$, an $s\in\{1, \dots, 2k-2x\}$ and an $\varepsilon\in\{0, 1\}$ such that $E^{m}(2k+1, (2k+1)(n+f)+j)$ is equivalent to $D^{s}(2k+1-2x, (2k+1-2x)(n+f)+h)\sqcup U_{\varepsilon}$, where $U_{0}$ is a circle in the plane and $U_{1}$ is the empty set, 
\item if $E^{m}(2k+1, (2k+1)(n+f)+j)$ is either type-$1$ or type-$2$, then $E^{m}(2k+1, (2k+1)(n+f)+j)$ has at most $Y_{2}$ negative crossings,  
\item if $E^{m}(2k+1, (2k+1)(n+f)+j)$ is either type-$3$ or type-$4$, then $E^{m}(2k+1, (2k+1)(n+f)+j)$ has at most $Y'_{2}$ negative crossings. 
\end{itemize}
From Lemma~$\ref{lem_cable}$, if $i-n_{-}+l_{-}(2k+1-2x)^{2}>2(k-x)(k-x+1)(n-l+1)+l(2k+1-2x)^{2}$ and $n\geq l$, then we have 
\begin{align*}
&H^{i}(E^{m}(2k+1, (2k+1)(n+f)+j))\\
&=H^{i-n_{-}+l_{-}(2k+1-2x)^{2}}(D^{s}(2k+1-2x, (2k+1-2x)(n+f)+h))=0, 
\end{align*}
where $n_{-}$ is the number of the negative crossings of $E^{m}(2k+1, (2k+1)(n+f)+j)$. 
In particular, if $i>2(k-x)(k-x+1)(n-l+1)+l_{+}(2k+1-2x)^{2}+n_{-}$ and $n\geq l$, then we have 
\begin{align*}
H^{i}(E^{m}(2k+1, (2k+1)(n+f)+j))=0.  
\end{align*}
Then we can prove the following claim. 
\begin{claim}\label{last}
For $j=1, \dots, 2k$ and $m=1, \dots, 2k$, 
if $E^{m}(2k+1, (2k+1)(n+f)+j)$ is either type-$1$ or type-$2$, then 
\begin{align}
&l(2k+1)^{2}+2k(k+1)(n-l)-2\label{claim4-(1)}\\
&\geq 2(k-x)(k-x+1)(n-l+1)+l_{+}(2k+1-2x)^{2}+Y_{2} \nonumber\\
&\geq 2(k-x)(k-x+1)(n-l+1)+l_{+}(2k+1-2x)^{2}+n_{-}, \nonumber
\end{align}
and if $E^{m}(2k+1, (2k+1)(n+f)+j)$ is either type-$3$ or type-$4$, then 
\begin{align}
&l(2k+1)^{2}+2k(k+1)(n-l)-2 \label{claim4-(2)}\\
&\geq 2(k-x)(k-x+1)(n-l+1)+l_{+}(2k+1-2x)^{2}+Y'_{2}\nonumber\\
&\geq 2(k-x)(k-x+1)(n-l+1)+l_{+}(2k+1-2x)^{2}+n_{-}. \nonumber
\end{align}
\end{claim}
We prove Claim~$\ref{last}$ latter. 
From the above discussion and Claim~$\ref{last}$,  if $i>l(2k+1)^{2}+2k(k+1)(n-l)-2$, then $H^{i}(E^{m}(2k+1, (2k+1)(n+f)+j))=0$ for $j=1, \dots, 2k$ and $m=1, \dots, 2k$.  
Now there is the following exact sequence: 
\begin{multline*}
\rightarrow H^{i-1}(E^{m}(2k+1, (2k+1)(n+f-1)+j))\rightarrow H^{i}(D^{m-1}(2k+1, (2k+1)(n+f-1)+j)) \\
\rightarrow H^{i}(D^{m}(2k+1, (2k+1)(n+f-1)+j))\rightarrow H^{i}(E^{m}(2k+1, (2k+1)(n+f-1)+j))\rightarrow , 
\end{multline*}
where $m=1, \dots, 2k$, $n\geq 0$ and $j=1, \dots, 2k$. 
From the above result and this exact sequence, we obtain 
\begin{align*}
&H^{2k(k+1)(n+l)+l}(D(2k+1, (2k+1)(n+f)-1))\\
&=H^{2k(k+1)(n+l)+l}(D^{1}(2k+1, (2k+1)(n+f-1)+2k-1))\\
&=\cdots= \\
&=H^{2k(k+1)(n+l)+l}(D^{2k}((2k+1, (2k+1)(n+f-1)+2k-1))\\
&=H^{2k(k+1)(n+l)+l}(D^{0}(2k+1, (2k+1)(n+f-1)+2k-2))\\
&=\cdots= \\
&=H^{2k(k+1)(n+l)+l}(D(2k, 2k(n+f-1))). \\ 
\end{align*}
\end{proof}

\begin{proof}[Proof of Claim~$\ref{last}$]
We have already proved $(\ref{claim4-(2)})$ in the proof of Claim~$\ref{key_claim2}$. Let us prove $(\ref{claim4-(1)})$. 
Recall $j=1, \dots, 2k+1$, $b\leq x\leq k$ and $x\geq 1$. 
Hence if $j\leq 2k-1$ or $x\geq 2$, we obtain 
\begin{multline*}
l(2k+1)^{2}+2k(k+1)(n-l)-2-(2(k-x)(k-x+1)(n-l+1)+l_{+}(2k+1-2x)^{2}+Y_{2})\\
=-2+x(2k+1-j)+2(k-x)(b-x)+x-b\geq0. \\
\end{multline*} 
%
If $j=2k$ and $x=1$, then $E^{m}(2k+1, (2k+1)(n+f)+j)$ is either type-$3$ or type-$4$. 
Hence if $E^{m}(2k+1, (2k+1)(n+f)+j)$ is either type-$1$ or type-$2$, we obtain 
$l(2k+1)^{2}+2k(k+1)(n-l)-2\geq 2(k-x)(k-x+1)(n-l+1)+l_{+}(2k+1-2x)^{2}+Y_{2}$ 
for $j=1, \dots, 2k$. 
\end{proof}

{\bf Acknowledgements:} The author would like to express his gratitude to Hitoshi ~Murakami for his encouragement. 
He also would like to thank Yuanyuan ~Bao for her helpful comments. 
K{\'a}lm{\'a}n Tam{\'a}s gives me constructive comments and warm encouragement. 
This paper is the author's master thesis at Tokyo Institute of Technology in February 2012. 
\bibliographystyle{amsplain}
\bibliography{mrabbrev,tagami}
\end{document}